\documentclass{article} 
\usepackage[T1]{fontenc}  
\usepackage{multicol}
\usepackage{color}  
\usepackage{amsmath}
\usepackage{titlesec}
\usepackage{mathrsfs}
\usepackage{amscd}
\usepackage{graphics}
\usepackage{amssymb} 
\usepackage{amsthm}                   
\usepackage{lastpage}         
 \usepackage{paralist}
\usepackage{stmaryrd}
\usepackage{ulem}
\usepackage{bbm}
\usepackage{faktor}
\usepackage{float}
\usepackage{oldgerm}
\usepackage{makeidx}
\usepackage{tikz-cd}
\usetikzlibrary{matrix,arrows,intersections,shapes.arrows,fit}
\usepackage[hypertexnames=false]{hyperref}
\usepackage{comment}

\usepackage[margin=1.5in]{geometry}

\makeindex
\titleformat{\section}[block]{\Large \sc\center}{\thesection.}{.5em}{}

\newcommand{\R}[0]{\mathbb{R}}

\renewcommand{\P}[0]{\mathbb{P}}

\newcommand{\N}[0]{\mathbb{N}}
\newcommand{\Q}[0]{\mathbb{Q}}
\newcommand{\Z}[0]{\mathbb{Z}}

\newcommand{\0}{\mathbf{0}}
\newcommand{\X}[0]{{\mathbf x}}
\newcommand{\bepsilon}[0]{{\boldsymbol \epsilon}}

\newcommand{\z}{\mathbf{z}}
\renewcommand{\a}{\mathbf{a}}

\renewcommand{\t}[0]{{\mathbf t}}
\newcommand{\K}[0]{{\mathbf k}}
\newcommand{\J}[0]{{\mathbf j}}
\renewcommand{\v}{\mathbf{v}}
\newcommand{\del}[0]{\partial}
\renewcommand{\i}{\mathrm{i}}

\newcommand{\D}[0]{\mathrm{d}}

\newcommand{\zz}[0]{\mathrm{z\kern-.3em\raise-0.4ex\hbox{z}}}

\newcommand{\bigO}[1]{O \left( #1 \right)}

\newcommand{\subgauss}[1]{\left\lfloor#1\right\rfloor}

\newcommand{\abs}[1]{\left\vert#1\right\vert}

\newcommand{\idl}[1]{\mathfrak{#1}}
\renewcommand{\hat}{\widehat}
\renewcommand{\backslash}{\setminus}

\newcommand{\xcrit}{\hat{\X}_{\J; \hat{\K}}}
\newcommand{\xcrita}{\hat{\X}_{\J; \hat{\a}}}
\newcommand{\xcritd}{\hat{\X}_{d; \J; \hat{\K}}}

\DeclareMathOperator{\rank}{rank}

\DeclareMathOperator{\supp}{supp}

\DeclareMathOperator{\dist}{dist}

\DeclareMathOperator{\Jac}{Jac}

\newtheorem{thm}{Theorem}[section]
\newtheorem{cor}[thm]{Corollary}
\newtheorem{lem}[thm]{Lemma}
\newtheorem{prop}[thm]{Proposition}

\newtheorem{cond}[thm]{Condition}
\newtheorem{conjecture}[thm]{Conjecture}

\theoremstyle{definition}

\newtheorem*{xrem}{Remark}

\makeatletter
\def\th@plain{
  \thm@notefont{}   
  \itshape
}
\def\th@definition{
  \thm@notefont{}
  \normalfont 
}
\makeatother

\numberwithin{equation}{section}

\title{\bf On the number of rational points close to a compact manifold under a less restrictive curvature condition}
\author{\sc Florian Munkelt}
\date{}
\begin{document}\noindent
\maketitle

\renewcommand{\abstractname}{}
\begin{abstract}
{\sc Abstract.} Let $\mathscr{M}$ be a compact submanifold of $\R^{M}$. In this article we establish an asymptotic formula for the number of rational points within a given distance to $\mathscr{M}$ and with bounded denominators under the assumption that $\mathscr{M}$ fulfills a certain curvature condition. Our result generalizes earlier work from Schindler and Yamagishi \cite{Dam}, as our curvature condition is a relaxation of that used by them. We are able to recover a similar result concerning a conjecture by Huang and a slightly weaker analogue of Serre's dimension growth conjecture for compact submanifolds of $\R^{M}$.
\end{abstract}
\tableofcontents
\section{Introduction}
Let $\mathscr{M}$ be a compact immersed submanifold of $\R^{M}$ and $R=M-\dim \mathscr{M}$ its codimension. Given an integer $Q\in \N$ and $\delta\geq 0$, we study the number \[N(\mathscr{M}; Q, \delta):=\# \left\{(\a, q)\in \Z^{M}\times \N\left|\, 1\leq q\leq Q, \dist\left(\frac{\a}{q}, \mathscr{M}\right)\leq \frac{\delta}{q}\right.\right\}\]
of rational points with denominators bounded by $Q$ and $L^{\infty}$-distance to $\mathscr{M}$ bounded by $\delta$.
The study of rational points lying 'close' to manifolds is an approach to the fundamental study of rational points on algebraic varieties. As such it has relevant applications to a number of problems within Diophantine geometry, besides being an interesting problem in its own right. See for example \cite{Elki}, \cite[\S2-\S5]{Hua1}, \cite{Hugh}, \cite{Maz} or \cite{Dam2}.

We can readily state the trivial estimate
\[N(\mathscr{M}; Q, \delta)\ll Q^{\dim\mathscr{M}+1},\]
and employing a probabilistic heuristic yields
\[\delta^{R}Q^{\dim \mathscr{M}+1}\ll N(\mathscr{M}; Q, \delta)\ll \delta^{R}Q^{\dim \mathscr{M}+1}.\]
It is known that this heuristic does not hold unconditionally. For example, if $\mathscr{M}$ is a rational hyperplane in $\R^{M}$ and $\delta \leq 1$, then we find that
\[Q^{\dim\mathscr{M}+1}\ll N(\mathscr{M}; Q, \delta)\ll Q^{\mathscr{M}+1}.\]
We additionally see from that example, that in order to establish non-trivial bounds we may be inclined to study manifolds with a 'proper curvature condition'. Huang proposed the following conjecture in his groundbreaking work \cite{Hua1}.

\begin{conjecture}\label{conj}
Let $\mathscr{M}$ be a bounded immersed submanifold of $\R^{M}$ with boundary. Let $R=M-\dim\mathscr{M}$ and suppose $\mathscr{M}$ satisfies a 'proper curvature condition'. Then there exists a constant $c_{\mathscr{M}}>0$ depending only on $\mathscr{M}$ such that
\[N(\mathscr{M}; Q, \delta)\sim c_{\mathscr{M}}\delta^{R}Q^{\dim \mathscr{M}+1}\]
when $\delta\geq Q^{-\frac{1}{R}+\epsilon}$ for some $\epsilon>0$ and $Q\rightarrow \infty$.
\end{conjecture}
It is not made explicit what 'proper curvature conditions' means in the given context.

The first non-trivial case that has been studied extensively is that of a compact curve in $\R^{2}$ with curvature bounded away from zero. In this setting, Huxley \cite{Hux} was the first to obtain a notable upper bound for a $\mathcal{C}^{2}$ curve $\mathscr{C}$, which has later been given in the version
\[N(\mathscr{C}; Q, \delta)\ll_{\mathscr{C}} \delta^{1-\epsilon} Q^{2}+Q\log Q\]
for any $\delta$ and $Q$ in \cite{Vaug}. In fact, Vaughan and Velani \cite{Vaug} showed that
\[N(\mathscr{C}; Q, \delta)\ll_{\mathscr{C}} \delta Q^{2}+Q^{1+\epsilon}\]
for a $\mathcal{C}^{3}$ curve $\mathscr{C}$, which is the upper bound that Conjecture \ref{conj} predicts.

Conversely, a  sharp lower bound has been established by Beresnevich, Dickinson and Velani \cite{Ber1}
\[\delta Q^{2}\ll _{\mathscr{C}}N(\mathscr{C}; Q, \delta)\]
with $\delta\gg Q^{-1}$ and $\delta Q\rightarrow \infty$, for a $\mathcal{C}^{3}$ curve $\mathscr{C}$ admitting at least one point with non-vanishing curvature. Further work by Huang \cite{Hua2} established an asymptotic formula for $\mathcal{C}^{3}$ curves. The interested reader may find more details on the case for planar curves in \cite{Hua2}.

For the case of general manifolds Beresnevich established the sharp lower bound
\[\delta^{k}Q^{\dim \mathscr{M}+1}\ll_{\mathscr{M}}N(\mathscr{M}; Q, \delta)\]
for any $\delta \gg Q^{-\frac{1}{R}}$, assuming $\mathscr{M}$ is an analytic submanifold of $\R^{M}$ which admits at least one non-degenerate point, in his spectacular work \cite{Ber2}.

Huang established Conjecture \ref{conj} in the case  when $\mathscr{M}$ is  a hypersurface with Gaussian curvature bounded away from zero in $\R^{M}$ in \cite{Hua1}.

A recent generalization of this result ist due to Schindler and Yamagishi \cite{Dam}, who established the Conjecture \ref{conj} in the case of a compact immersed submanifold of $\R^{M}$ in codimension $R$ with a curvature condition that reduces to Huang's case for $R=1$. In particular, if the manifold $\mathscr{M}$ is locally parametrized by the functions $f_{1}, ..., f_{R}$, the required curvature condition is as follows.

\begin{cond}\label{oldcurvcon}
Given any $\t\in \R^{R}\backslash \{0\}$, we have
\[\det H_{t_{1}f_{1}+\cdots +t_{R}f_{R}}(\X_{0})\neq 0,\]
where $H_{f}$ denotes the Hessian matrix of the function $f$.
\end{cond}

We continue by presenting the details of our main result. By the compact nature of $\mathscr{M}$, the argument reduces to a finite number of local arguments, hence we may assume without loss of generality that
\begin{equation}\label{manifolddef}
\mathscr{M}:=\{(\X, f_{1}(\X),..., f_{R}(\X))\in \R^{M}\mid \X=(x_{1},..., x_{n})\in \overline{B_{\varepsilon_{0}}(\X_{0})}\},
\end{equation}
where $\X_{0}\in \R^{n}$, $\varepsilon_{0}>0$ and $f_{r}\in \mathcal{C}^{\ell}(\R^{n})$ for $1\leq r\leq R$ and some $\ell\geq 2$. Note that this specifically means $\dim \mathscr{M}=n$.\\
Let $\omega\in \mathcal{C}^{\infty}_{c}(\R^{n})$ be a non-negative weight function that is compactly supported in a sufficiently small neighbourhood of $\X_{0}$ and define
\[N_{\omega}(Q, \delta)=\sum_{\substack {\a\in \Z^{n}\\ q\leq Q \\  ||qf_{r}(\a/q)||\leq \delta \\  1\leq r\leq R }}{\omega\left(\frac{\a}{q}\right)},\]
where $||\cdot||$ denotes the distance to the closest integer. Obviously $||x||\leq 1/2$ for any $x\in \R$, hence we only consider $0\leq \delta \leq 1/2$. Let
\[N_{0}:=\sum_{\substack{\a\in \Z^{n}\\ q\leq Q}}{\omega\left(\frac{\a}{q}\right)}.\]
For a given function $f\in \mathcal{C}^{2}(\R^{n})$ we denote by $H_{f}(\X)$ the Hessian matrix of $f$ evaluated at $\X$, i.e. the $n\times n$-matrix whose entries are $\frac{\del^{2}f}{\del x_{\mu}\del x_{\nu}}(\X)$ for $1\leq \mu, \nu\leq n$.\\
We use the following relaxed curvature condition throughout this article.

\begin{cond}\label{newcurvcon}
Given any $\t\in \R^{R}\backslash \{0\}$, we have
\[\rank H_{t_{1}f_{1}+\cdots +t_{R}f_{R}}(\X_{0})\geq n-1.\]
\end{cond}
With these notations we have the following result.

\begin{thm}\label{mainthm}
Let $n\geq 3$ and $\ell>\max\{n+1, \frac{n}{2}+4\}$. Suppose \ref{newcurvcon} holds and that $\varepsilon_{0}>0$ is sufficiently small. Then we have
\[|N_{\omega}(Q, \delta)-(2\delta)^{R}N_{0}|\ll\left\{\begin{array}{*{2}{c}} \delta^{\frac{(R-1)(n-1)}{n+1}}Q^{n+\frac{2}{n+1}}\mathscr{E}_{n-1}(Q)&\mbox{if $\delta \geq Q^{-\frac{n-1}{n+2R-1}}$,}\\ Q^{n-\frac{(R-1)n+1-3R}{n+2R-1}}\mathscr{E}_{n-1}(Q)&\mbox{if $\delta <Q^{-\frac{n-1}{n+2R-1}} $,} \end{array}\right.\]
where
\[\mathscr{E}_{n-1}(Q)=\left\{\begin{array}{*{2}{c}} \exp(\idl{c}_{1}\sqrt{\log Q})& \mbox{if $n=3$,}\\ (\log Q)^{\idl{c}_{2}}& \mbox{if $n\geq 4$,}\end{array}\right.\]
for some positive constants $\idl{c}_{1}$ and $\idl{c}_{2}$. These constants as well as the implicit constants only depend on $\mathscr{M}$ and $\omega$.
\end{thm}
Comparing our exponents to those obtained in \cite[Theorem 1.2]{Dam}, we have $n+\frac{2}{n+1}$ instead of $n$ in the first case and $n-\frac{(R-1)n+1-3R}{n+2R+1}$ instead of $n-\frac{(n-2)(R-1)}{n+2R-2}$ in the scond one. Note that the cases in \cite{Dam} are distinguished by the comparison of $\delta$ against $Q^{-\frac{n}{n+2R-2}}$. As expected the bounds under the less restrictive curvature condition are worse, yet they have similar growth for large $n$.\\
\ \\
By Poisson summation formula we find that $N_{0}=\sigma Q^{n+1}+\bigO{Q^{n}}$ for some positive constant $\sigma$ depending only on $\omega$ and $n$ (compare \cite[(6.2)]{Hua1}). Combining this with \ref{mainthm} yields
\[N_{\omega}(Q, \delta)=(2\delta)^{R}\sigma Q^{n+1}+\bigO{\delta^{\frac{(R-1)(n+1)}{n+3}}Q^{n+\frac{4}{n+3}}\mathcal{E}_{n-1}(Q)},\]
when $\delta\geq Q^{\frac{n-1}{n+2R-1}+\epsilon}$ for any $\epsilon>0$ sufficiently small. Following the arguments in \cite[Section 7]{Hua1}, we can approximate the characteristic function of $\overline{B_{\varepsilon_{0}}(\X_{0})}$ by smooth weight functions and obtain:
\begin{cor}
Let $\mathscr{M}$ be as in (\ref{manifolddef}), $n\geq 3$ and $\ell >\max\{n+1, \frac{n}{2}+4\}$. Suppose Condition \ref{newcurvcon} holds and that $\varepsilon_{0}>0$ is sufficiently small. Then there exists a constant $c_{\mathscr{M}}>0$ depending only on $\mathscr{M}$ such that
\[N(\mathscr{M}; Q, \delta)\sim c_{\mathscr{M}}\delta^{R}Q^{dim \mathscr{M}+1}\]
when $\delta\geq Q^{-\frac{n-1}{n+2R-1}+\epsilon}$ for any $\epsilon>0$ suffciently small and $Q\rightarrow \infty$.
\end{cor}
Note that 
\[Q^{-\frac{1}{R}}\geq Q^{-\frac{n-1}{n+2R-1}}\]
only holds for $R>1$ and $n\geq \frac{3R-1}{R-1}$, hence Conjecture \ref{conj} holds in those cases. In fact, the asymptotic formula is obtained beyond the range of $\delta$ that was conjectured in those cases.

If we let $\delta=0$, then our (weighted) counting function gives the (weighted) number of rational points with bounded denominators that lie on the manifold $\mathscr{M}$. Applying the arguments from \cite[pp. 2047]{Hua1} to Conjecture \ref{conj}, we obtain
\[N(\mathscr{M}; Q, 0)\ll Q^{\dim \mathscr{M}+\epsilon},\]
for any $\epsilon>0$ sufficiently small with a generally sharp upper bound. This can be interpreted as an analogue of Serre's famous dimension growth conjecture for projective varieties in the context of smooth submanifolds of $\R^{M}$.

\begin{conjecture}[Serre's dimension growth conjecture]
Let $X\subseteq \P^{M-1}_{\Q}$ be an irreducible projective varietiy of degree at least two defined over $\Q$. Let $N_{X}(B)$ be the number of rational points on $X$ with naive height bounded by $B$. Then
\[N_{X}(B)\ll B^{\dim X}(\log B)^{c}\]
for some constant $c>0$.
\end{conjecture}
The reader may be interested in the large amount of literature concerning the dimension growth conjecture and find \cite{Cast} to be a nice introduction to the topic. For further reading we may refer to several examples, such as \cite{Brob1}, \cite{Brob2}, \cite{Brow1}, \cite{Brow2}, \cite{Brow3}, \cite{Brow4}, \cite{HB}, \cite{Marm}, \cite{Salb1}, \cite{Salb2}, \cite{Salb3}, \cite{Walsh}. Our analogue result for smooth submanifolds of $\R^{M}$ is the following estimate.

\begin{cor}
Let $\mathscr{M}$ be as in (\ref{manifolddef}), $n\geq 4$ and $\ell>\max\{n+1, \frac{n}{2}+4\}$. Suppose Condition \ref{newcurvcon} holds and that $\varepsilon_{0}>0$ is sufficiently small. Then
\[N(\mathscr{M}; Q, 0)\ll Q^{n-\frac{(R-1)n+1-3R}{n+2R-1}}(\log Q)^{c}\]
for some constant $c>0$.
\end{cor}
Note that in contrast to the situation in \cite{Dam} we do not unconditionally break the $\dim \mathscr{M}$ barrier here, only if $n>\frac{3R-1}{R-1}$ and $R>1$.

We adapt the strategy for proving Theorem \ref{mainthm} established in \cite{Dam}, which relies on the methods developed by Huang in \cite{Hua1} and fibration arguments. In particular, Schindler and Yamagishi reduced the problem to that for one function, such that the main result of \cite{Hua1} can be used. This is achieved by a more complicated version of the procedure developed by Huang in \cite{Hua1}, which  relates the counting problem of a function to that of its Legendre transform, for a family of functions satisfying the curvature condition \ref{oldcurvcon} and apllying it twice. For our relaxed curvature condition \ref{newcurvcon} we can use a similar approach with some necessary adjustments to accomodate an additional degree of freedom. After collecting some preliminary results in Section \ref{prelim}, we discuss the setup of our proof for  Theorem \ref{mainthm} in Section \ref{setup}. Section \ref{secbounds} is dedicated to establishing some auxiliary bounds, one of which depending on a result which is proven in Section \ref{secess}. Lastly, we combine our findings to prove Theorem \ref{mainthm} in Section \ref{maintheoremsec}.

\section{Preliminaries}\label{prelim}
In this article we denote by $\mathcal{C}^{\ell}(\mathcal{V})$ the set of $\ell$-times continuously differentiable functions defined on an open set $\mathcal{V}\subseteq \R^{n}$. Analogously $\mathcal{C}^{\infty}(\mathcal{V})$ denotes the set of smooth functions defined on $\mathcal{V}$ and $\mathcal{C}^{\infty}_{c}(\mathcal{V})$ the set of smooth functions defined on $\mathcal{V}$ that have a  compact support. Given any $f\in \mathcal{C}^{1}(\R^{n})$ we let $\nabla f=(\frac{\del f}{\del x_{1}}, ..., \frac{\del f}{\del x_{n}})$ be the gradient of $f$. For a subset $X\subseteq \R^{n}$ we denote the boundary of $X$ by $\del X=\overline{X}\backslash X^{\circ}$, where $\overline{X}$ denotes the closure of $X$ and $X^{\circ}$ denotes the interior of $X$. For any $z\in \R$ we let $e(z)=e^{2\pi \i z}$ and $||z||$ denotes the distance to the closest integer. For $\z=(z_{1}, ..., z_{n})\in \R^{n}$ let $|\z|=\max_{1\leq i\leq n}|z_{i}|$ denote the $L^{\infty}$-Norm and given any $\varepsilon >0$ we let 
\[B^{n}_{\varepsilon}(\z)=\{\X\in \R^{n}\mid |\X-\z|<\varepsilon\}=(z_{1}-\varepsilon, z_{1}+\varepsilon)\times \cdots \times (z_{n}-\varepsilon, z_{n}+\varepsilon).\] We may write $B_{\varepsilon}(\z)$ instead of $B^{n}_{\varepsilon}(\z)$ if the dimension is clear from context. By the notation $f(\X)\ll g(\X)$ or $f=\bigO{g(\X)}$ we mean that there exists a constant $C>0$ such that $|f(\X)|\leq Cg(\X)$ for all $\X$ in consideration.\\
\\ \ 
Given $F\in \mathcal{C}^{\ell}(\R^{n})$ let $\mathcal{U}\subseteq \R^{n}$ be an open subset such that $\nabla F$ is invertible on $\mathcal{U}$. We define the Legendre transform $F^{\ast}\colon \nabla F(\mathcal{U})\rightarrow \R$ of $F$ by
\[F^{\ast}(\z)=\z\cdot (\nabla F)^{-1}(\z)-(F\circ (\nabla F)^{-1})(\z).\]
It can be verified that $F^{\ast}$ is $\ell$-times continuously differentiable, $F^{\ast\ast}=F$ and $\nabla F^{\ast}=(\nabla F)^{-1}$. If $\X=\nabla F(\z)$ we obviously have
\begin{equation}\label{legendrepoint}
F^{\ast}(\X)=\X\cdot \z-F(\z)
\end{equation}
and furthermore
\begin{equation}\label{legendrematrix}
H_{F^{\ast}}(\X)=H_{F}(\z)^{-1}.
\end{equation}
For the following results on oscillatory integrals, we refer the reader to see for example \cite[Theorem 7.71 and Theorem 7.7.5]{Hor}.

\begin{lem}[non-stationary phase]\label{nonstat}
Let $\ell\in \N$ and $U_{+}\subseteq \R^{d}$ a bounded open set. Let $\omega\in \mathcal{C}_{c}^{\ell-1}(\R^{d})$ with $\supp\omega\subseteq U_{+}$ and $\varphi\in \mathcal{C}^{\ell}(U_{+})$ with $\nabla\varphi(\X)\neq \textbf{0}$ for all $\X\in \supp\omega$. Then for any $\lambda>0$
\[\left|\int_{\R^{d}}{\omega(\X)e(\lambda\varphi(\X))\D\X}\right|\leq c_{\ell}\lambda^{-\ell+1},\]
where the constant $c_{\ell}$ only depends on $\ell, d$, upper bounds for the absolute values of finitely many derivatives of $\omega$ and $\varphi$ on $U_{+}$ and a lower bound for $|\nabla\varphi|$ on $\supp\omega$.
\end{lem}

Recall given a symmetric matrix we define its signature to be the number of positive eigenvalues minus the number of negative eigenvalues.

\begin{lem}[stationary phase]\label{stat}
Let $\ell>\frac{d}{2}+4$ and $\mathscr{D}, \mathscr{D}_{+} \subseteq \R^{d}$ bounded open sets such that $\overline{\mathscr{D}}\subseteq \mathscr{D}_{+}$. Let $\omega\in \mathcal{C}_{c}^{\ell-1}(\R^{d})$ with $\supp\omega\subseteq \mathscr{D}$ and $\varphi\in \mathcal{C}^{\ell}(\mathscr{D})$. Suppose $\nabla\varphi(\v_{0})= \0$ and $H_{\varphi}(\v_{0})\neq 0$ for some $\v_{0}\in \mathscr{D}$. Let $\sigma$ be the signature of $H_{\varphi}(\v_{0})$ and $\Delta=|\det H_{\varphi}(\v_{0})|$. Suppose further that $\nabla \varphi(\X)\neq \0$ for all $\X\in \overline{\mathscr{D}}\backslash \{\v_{0}\}$. Then for any $\lambda>0$
\[\int_{\R^{d}}{\omega(\X)e(\lambda\varphi(\X))\D\X}= e\left(\lambda\varphi(\v_{0})+\frac{\sigma}{8}\right)\Delta^{-\frac{1}{2}}\lambda^{-\frac{d}{2}}(\omega(\v_{0})+\bigO{\lambda^{-1}}),\]
where the implicit constant only depends on $\ell, d$, upper bounds for the absolute values of finitely many derivatives of $\omega$ and $\varphi$ on $\mathscr{D}_{+}$, an upper bound for $|\X-\v_{0}|/|\nabla\varphi(\v_{0})|$ on $\mathscr{D}_{+}$, and a lower bound for $\Delta$.
\end{lem}
Note that this is a simplified version of \cite[Theorem 7.7.5]{Hor}, where the assumption on $\ell$ can be deduced from \cite[pp. 222, Remark]{Hor}.\\ \
\\ We consider two compactness resulst. Let $m\in \N_{0}$ and $\mathcal{G}=\mathcal{G}_{1}\times \mathcal{G}_{2}\subseteq \R^{n+m}$, where $\mathcal{G}_{1}\subseteq \R^{n}$ and $\mathcal{G}_{2}\subseteq \R^{m}$ are bounded connected open sets. Let $\X_{0}$ be a fixed point in $\mathcal{G}_{1}$.

\begin{lem}{}\label{compact1}
Let $G\in \mathcal{C}^{\ell}(\mathcal{G})$, $\ell\geq 2$ and assume $H_{G_{\t}}(\X)\neq 0$ for every $\t\in \mathcal{G}_{2}$, where $G_{\t}$ is the real-valued map on $\mathcal{G}_{1}$ given by $\X\mapsto G(\X, \t)$. Let $\mathcal{F}_{2}$ be a compact set contained in $\mathcal{G}_{2}$, then there exist a real number $\tau>0$ and constants $c_{1}, c_{2}>0$ such that
\[c_{1}\leq |\det H_{G_{\t}}(\X)|\leq c_{2}\]
for all $\X\in B_{\tau}(\X_{0})$ and $\t\in \mathcal{F}_{2}$. Moreover the map $\X\mapsto \nabla G_{\t}(\X)$ is a $\mathcal{C}^{\ell-1}$-diffeomorphism on $\overline{B_{\tau}(\X_{0})}$ for all $\t\in \mathcal{F}_{2}$.
\end{lem}

\begin{lem}\label{compact2}
Let $G\in \mathcal{C}^{\ell}(\mathcal{G})$, $\ell\geq 2$  and assume $H_{G_{\t}}(\X)\neq 0$ for every $\t\in \mathcal{G}_{2}$. Let $\mathcal{F}_{2}$ and $\tau$ be as in Lemma \ref{compact1}. Then for any $0<\kappa<\tau$ sufficiently small, there exists $\rho>0$ such that
\[\dist(\del (\nabla G_{\t}(B_{\tau}(\X_{0}))), \del (\nabla G_{\t}(B_{\kappa}(\X_{0}))))\geq 2\rho\]
for all $\t\in \mathcal{F}_{2}$.
\end{lem}
For proofs of \ref{compact1} and \ref{compact2} we refer the reader to \cite[Lemmas 3.4 and 3.5]{Dam}.

\section{Setting up the proof of \ref{mainthm}}\label{setup}
By virtue of the characteristic functions
\begin{equation}\label{indi}
\chi_{\delta}(\theta)=\left\{ \begin{array}{*{2}{c}}1& \mbox{if $||\theta||\leq \delta$},\\ 0& \mbox{else.}\end{array}\right.
\end{equation}
for $0< \delta \leq 1/2$ we can rewrite
\[N_{\omega}(Q, \delta)=\sum_{\substack {\textbf{a}\in \Z^{n}\\ q\leq Q}}{\omega\left(\frac{\textbf a}{q}\right)}\prod_{r=1}^{R}{\chi_{\delta}\left(qf_{r}\left(\frac{\textbf {a}}{q}\right)\right)}.\]
Consider the Selberg magic functions as described in \cite{Mont}
for the interval $[-\delta, \delta]\subseteq \R/\Z$ and a parameter $J\in \N$
\[S_{J}^{\pm}(x)=\sum_{|j|\leq J}{\hat{S}_{J}^{\pm}(j)e(jx)}.\]
They obey the properties
\[S^{-}_{J}(y)\leq \chi_{\delta}(y)\leq S^{+}_{J}(y)\]
and
\[\hat{S}^{\pm}_{J}(0)=2\delta\pm \frac{1}{J+1}\]
and are bounded by
\begin{equation}\label{magicbound}
|\hat{S}^{\pm}_{J}(j)|\leq \frac{1}{J+1}+\min\left(2\delta, \frac{1}{\pi|j|}\right)
\end{equation}
for all $y\in \R/\Z$ and $0\leq |j|\leq J$. Hence we can bound the characteristic functions from above by the Selberg magic functions and obtain
\begin{align*}
N_{\omega}(Q, \delta)&\leq \sum_{\substack{\a\in \Z^{n}\\ q\leq Q}}{\omega\left(\frac{\a}{q}\right)\prod_{r=1}^{R}{S^{+}_{J}\left(qf_{r}\left(\frac{\a}{q}\right)\right)}}\\
&= \sum_{\substack{\a\in \Z^{n}\\ q\leq Q}}{\omega\left(\frac{\a}{q}\right)\prod_{r=1}^{R}{\left(\sum_{j_{r}=-J}^{J}{\hat{S}^{+}_{J}(j_{r})e\left(j_{r}qf_{r}\left(\frac{\a}{q}\right)\right)}\right)}}\\
&= \sum_{\substack{\a\in \Z^{n}\\ q\leq Q}}{\omega\left(\frac{\a}{q}\right)\sum_{\substack{0\leq |j_{i}|\leq J\\ 1\leq i\leq R}}{\left(\prod_{r=1}^{R}{\hat{S}^{+}_{J}(j_{r})}\right)e\left(\sum_{r=1}^{R}{j_{r}qf_{r}\left(\frac{\a}{q}\right)}\right)}}.
\end{align*}
The terms with $j_{i}=0$ for all $i=1,..., R$ contribute
\[\left(2\delta+\frac{1}{J+1}\right)^{R}N_{0}=(2\delta)^{R}N_{0}+\bigO{\delta^{R-1}\frac{Q^{n+1}}{J}+\frac{Q^{n+1}}{J^{R}}},\]
with the implicit constant possibly depending on $R$ and an upper bound for the diameter of $\supp w$. Bounding the characteristic function from below by the Selberg magic functions yields a similar result, such that we conclude
\begin{align}\label{maineq}
|N_{\omega}(Q, \delta)-(2\delta)^{R}N_{0}| &\ll \delta^{R-1}\frac{Q^{n+1}}{J}+\frac{Q^{n+1}}{J^{R}}\\
&\quad +\sum_{\substack{1\leq |j_{i}|\leq J\\ 1\leq i\leq R\\ \J\neq \textbf{0}}}{\left(\prod_{r=1}^{R}{b_{j_{r}}}\right)\left|\sum_{\substack{\a\in \Z^{n}\\ q\leq Q}}{\omega\left(\frac{\a}{q}\right)e\left(\sum_{r=1}^{R}{j_{r}qf_{r}\left(\frac{\a}{q}\right)}\right)}\right|},\nonumber
\end{align}
where 
\[b_{j_{r}}:=\frac{1}{J+1}+\min\left(2\delta, \frac{1}{\pi|j_{r}|}\right)\] 
is the bound for the Selberg magic functions given in (\ref{magicbound}).
Via Poisson summation formula we can rewrite
\begin{align}
&\sum_{\a\in \Z^{n}}{\omega\left(\frac{\a}{q}\right)e\left(\sum_{r=1}^{R}{j_{r}qf_{r}\left(\frac{\a}{q}\right)}\right)}\\
&=\sum_{\K\in \Z^{n}}{\int_{\R^{n}}{\omega\left(\frac{\X}{q}\right)e\left(\sum_{r=1}^{R}{j_{r}qf_{r}\left(\frac{\X}{q}\right)-\K\cdot\X}\right)\D\X}}\nonumber\\
&=q^{n}\sum_{\K\in \Z^{n}}{I(q; \J; \K)}\nonumber
\end{align}
with
\[I(q; \J; \K)=\int_{\R^{n}}{\omega(\X)e\left(\sum_{r=1}^{R}{qj_{r}f_{r}(\X)-q\K\cdot \X}\right)\D\X}.\]
We now make use of \ref{newcurvcon}. By asssumption the Hessian matrix $H_{t_{1}f_{1}+\cdots +t_{r}f{r}}(\X_{0})$ has a non-vanishing minor of size $n-1$, i.e. we can find $1\leq i, j\leq n$ such that after deleting the $i$-th row and the $j$-th column of $H_{t_{1}f_{1}+\cdots + t_{t}f_{r}}(\X_{0})$, the resulting matrix is invertible. Since the Hessian matrix is symmetric, we necessarily have $i=j$ in this case. Consider the functions
\[\vartheta_{i}\colon \R^{R}\backslash\{0\}\rightarrow \R, \t\mapsto \det \left(\frac{\del^{2}t_{1}f_{1}+\cdots +t_{R}f_{R}}{\del x_{\nu}\del x_{\mu}}\right)_{\substack{1\leq \nu, \mu\leq n\\ \nu,\mu\neq i}}(\X_{0})\]
for $1\leq i\leq n$. Given \ref{newcurvcon} the preimages $\{\vartheta^{-1}_{i}(\R\backslash\{0\})\}_{1\leq i\leq n}$ form an open cover of $\R^{R}\setminus\{0\}$. For any fixed $\t$ the rank condition is invariant under scaling with a linear factor $a\in \R$., i.e. if $\t$ belongs to $\vartheta^{-1}_{i}(\R\setminus\{0\})$ so does $a\t=(at_{1}, ..., at_{R})$. Therefore we can assume $\t$ to be normalized in the sense $|\t|=1$. For every $1\leq i\leq n$ let $\tilde{T}_{i}=\vartheta_{i}^{-1}(\R\backslash\{0\})\cap B_{1}^{R}(\0) $, then $\{\tilde{T}_{1}, ..., \tilde{T}_{n}\}$ is an open cover of $B_{1}^{R}(\0)$. Since $B_{1}^{R}(\0)$ is compact and Hausdroff, it is also normal, hence the open cover $\{\tilde{T}_{1}, ..., \tilde{T}_{n}\}$ admits a shrinkage. That is an open cover $\{T_{1}', ..., T_{n}'\}$ such that $T_{i}:=\overline{T_{i}'}\subseteq \tilde{T}_{i}$ for $1\leq i \leq n$.\\ \
\\
To find a bound for the last term in \ref{maineq} it suffices to find an upper bound for
\begin{equation}\label{NrQd}
N^{(r; \bepsilon; \nu)}(Q, \delta)=\sum_{\substack{1\leq j_{r}\leq J\\ 0\leq j_{s}\leq j_{r}\\ (\J/j_{r})\in T_{\nu}}}{\left(\prod_{r=1}^{R}{b_{j_{r}}}\right)\left|\sum_{q\leq Q}{q^{n}}\sum_{\textbf{k}\in \Z^{n}}{I(q; (\epsilon_{1}j_{1}, ..., \epsilon_{R}j_{R}); \K)}\right|}
\end{equation}
for each $1\leq r\leq R$, $\bepsilon\in \{-1, 1\}^{R}$ and $1\leq \nu\leq n$. The arguments turn out to be identical for all $(r; \bepsilon; \nu)$, since different choices of $r$ or $\epsilon$ admit only to relabeling and the choice of $\nu$ is merely an exercise in notation. Therefore we only present the details for $N^{(1; (1, ..., 1); n)}(Q, \delta)$. Note that the same upper bound in fact holds for all $N^{(r; \bepsilon; n)}$.\\ \
\\
To $\X=(x_{1}, ..., x_{n})\in \R^{n}$ let $\hat{\X}=(x_{1}, ..., x_{n-1})\in \R^{n-1}$ and define functions
\[\hat{f}_{j, y}\colon \R^{n-1}\rightarrow \R, \hat{\X}\mapsto f_{j}(\hat{\X}, y)\]
for $y$ in a sufficiently small neighborhood $U(x_{0, n})$ of $x_{0, n}$. Define further
\[\hat{G}_{y}(\hat{\X}, \t)=\hat{f}_{1, y}(\hat{\X})+\sum_{r=2}^{R}{t_{r}\hat{f}_{r, y}(\hat{\X})}\]
and consider the continuous function
\[\psi\colon U(x_{0, n})\rightarrow \R, y\mapsto \det \left(\frac{\del^{2}(f_{1}+t_{2}f_{2+}\cdots +t_{R}f_{R})}{\del x_{\nu}\del x_{\mu}}\right)_{1\leq \nu, \mu\leq n-1}(\hat{\X}, y).\]
For a suitable $\varepsilon'>0$ we have $\psi(y)\neq 0$ for $y\in B_{\varepsilon}(x_{0, n})$, since $\psi(x_{0, n})\neq 0$ by construction. Take $0<\varepsilon_{1}<\varepsilon'$ sufficiently small, then on the compact set $\overline{\mathscr{Y}}$ with $\mathscr{Y}=B_{\varepsilon_{1}}(x_{0, n})$ we have
\[c'_{1}\leq \left|\det \left(\frac{\del^{2}(f_{1}+t_{2}f_{2}+\cdots +t_{R}f_{R})}{\del x_{\nu}\del x_{\mu}}\right)_{1\leq \nu, \mu\leq n-1}(\hat{\X}_{0}, y)\right|\leq c'_{2}\]
with constants $0<c'_{1}, c'_{2}$ for all $\t\in T_{n}$.
Now $\hat{G}_{y}$ satisfies the conditions of lemmas \ref{compact1} and $\ref{compact2}$ for $\mathcal{F}_{2}=\overline{\mathscr{Y}}\times T_{n}$, i.e. there are constants $\tau_{(1; (1, ..., 1); n)}>0$ and $c_{1}, c_{2}>0$ such that
\begin{equation} \label{determinantbound}
c_{1}\leq\left|\det \left(\frac{\del^{2}(f_{1}+t_{2}f_{2}\cdots +t_{R}f_{R})}{\del x_{\nu}\del x_{\mu}}\right)_{1\leq \nu, \mu\leq n-1}(\hat{\X}, y)\right|\leq c_{2}
\end{equation}
for all $\textbf{t}\in T_{n}$, $y\in \overline{\mathscr{Y}}$ and $\X\in B_{2\tau_{(1; (1, ..., 1); n)}}(\hat{\X}_{0})$. Moreover, the map
\[\hat{\X}\mapsto \left(\hat{f}_{1, y}+\sum_{r=2}^{R}{t_{r}\hat{f}_{r, y}}\right)(\hat{\X})\]
is a $\mathcal{C}^{\ell-1}$ diffeomorphism on $\overline{B_{2\tau_{(1; (1,..., 1); n)}}(\hat{\X}_{0})}$ for all $\t\in T_{n}$ and $y\in \overline{\mathscr{Y}}$. Define $\tau_{(r; \epsilon; \nu)}$ in the same way for $1\leq r\leq R$, $\epsilon\in \{\pm 1\}^{R}$ and $1\leq \nu \leq n$ and let
\[0\leq \tau\leq \min_{\substack{1\leq r\leq R\\ \epsilon\in \{\pm 1\}^{R}\\ 1\leq \nu \leq n}}\tau_{(r; \epsilon; \nu)}\]
be sufficiently small such that \ref{nonstatgradientbound} is applicable. For this choice of $\tau$ with \ref{compact2} we find constants $0<\kappa <\tau$ and $\rho$ such that
\begin{equation}\label{rho}
\dist\left(\del\left(\nabla\left(\hat{f}_{1, y}+\sum_{r=2}^{R}{t_{r}\hat{f}_{r, y}}\right)(B_{\tau}(\hat{\X}_{0}))\right), \del\left(\nabla\left(\hat{f}_{1, y}+\sum_{r=2}^{R}{t_{r}\hat{f}_{r, y}}\right)(B_{\kappa}(\hat{\X}_{0}))\right)\right)\geq 2\rho
\end{equation}
for all $\t\in T_{n}$ and $y\in \overline{\mathscr{Y}}$. Note that $\varepsilon_{0}<2\tau$ is a sufficient choice in \ref{mainthm}.

Let $\mathscr{D}=B_{\tau}(\hat{\X}_{0})$ and let $\omega\in \mathcal{C}_{c}^{\infty}(\R^{n})$ be a non-negative weight function such that for any $y\in \mathscr{Y}$ the closure of
\[U_{y}:=\{\X\in \R^{n-1}\mid \omega(\X, y)\neq 0\}\]
is contained in $B_{\kappa}(\hat{\X}_{0})$. Define $\hat{F}_{y, \textbf{j}}=\hat{f}_{1, y}+(j_{2}/j_{1})\hat{f}_{2, y}+\cdots + (j_{R}/j_{1})\hat{f}_{R, y}$ and $V_{y, \textbf{j}}=\nabla \hat{F}_{y, \textbf{j}}(U_{y})$. 
Since $0\leq j_{r}/j_{1}\leq 1$ for $2\leq r\leq R$ we know that $\nabla \hat{F}_{y, \textbf{j}}$ is a diffeomorphism on $U_{y}$ and $\mathscr{D}$.

\begin{lem}\label{boundedfy}
The functions $\hat{f}_{r, y}$ for $1\leq r\leq R$ are bounded on $\overline{B_{2\tau}(\hat{\X}_{0})}$ for all $y\in \overline{\mathscr{Y}}$ and the bounds are independend of $y$. Additionally, there is $L\in \N$ such that for all $\t\in T_{n}$ and all $i_{1}, ..., i_{n-1}\in \Z_{\geq 0}$with $\sum_{\mu=1}^{n}{i_{\mu}}\leq \ell$ we have
\[\left|\frac{\del^{i_{1}+\cdots + i_{n-1}}t_{1}f_{1}+\cdots t_{R}f_{R}}{\del^{i_{1}}_{x_{1}}\cdots \del^{i_{n-1}}_{x_{n-1}}}(\X)\right|\leq L\]
on $\overline{B_{2\tau}(\hat{\X}_{0})}\times\overline{\mathscr{Y}}$ and $\rho\leq L$ for $\rho$ in (\ref{rho}).
\end{lem}

\begin{proof}
By assumption $f_{r}(\X)$ is smooth, hence on the compact domain $\overline{B_{2\tau}(\hat{\X}_{0})}\times\overline{\mathscr{Y}}$ it attains a maximum $M_{r}$. Now by definition $\hat{f}_{r, y}(\hat{\X})=f_{r}(\hat{\X}, y)$, hence for any $y\in \overline{\mathscr{Y}}$
\[|\hat{f}_{r, y}(\hat{\X})|<M_{r}\]
on $\overline{B_{2\tau}(\hat{\X}_{0})}$.
For the derivatives note that all domains of definition, i.e. $\overline{B_{2\tau}(\hat{\X}_{0})}$, $\overline{\mathscr{Y}}$ and $T_{n}$ are compact and all the relevant functions depend at least continuously on $\X=(\hat{\X}, y)$ and $\t$, hence for any given suitable $(i_{1}, ..., i_{n-1})$ there exists a maximum
\[M_{(i_{1}, ..., i_{n-1})}=\max_{\X, \t}{\left|\frac{\del^{i_{1}+\cdots + i_{n-1}}t_{1}f_{1}+\cdots t_{R}f_{R}}{\del^{i_{1}}_{x_{1}}\cdots \del^{i_{n-1}}_{x_{n-1}}}(\X)\right|}.\]
Since there are only finitely many suitable choices for $(i_{1}, ..., i_{n-1})$ we can also take the maximum over them and define $L$ to be the smalles natural number such that
\[L\geq \max_{(i_{1}, ..., i_{n-1})}M_{(i_{1}, ..., i_{n-1})}\]
and
\[L\geq \rho.\]
\end{proof}

Specifically we have that $V_{y, \J}\subseteq [-L, L]^{n-1}$ independently of $y$ and $\J$.\\ 
\ \\
We split the set of $\K\in \Z^{n}$ into three disjoint subsets. Let $\hat{\K}=(k_{1}, ..., k_{n-1})$ and  
\[D(\hat{\K}, \J)=\min_{y\in\overline{\mathscr{Y}}}\dist\left(\frac{\hat{\K}}{j_{1}}, V_{y, \J}\right).\]
Now define
\[\mathscr{K}_{\J; 1}=\left\{\K\in \Z^{n}\left| \frac{\hat{\K}}{j_{1}}\in \bigcup_{y\in\overline{\mathscr{Y}}}{V_{y, \J}}, |k_{n}|\leq 2j_{1}L\right.\right\},\]
\[\mathscr{K}_{\J; 2}=\left\{\K\in \Z^{n}\left|D(\hat{\K}, \J)\geq \rho \right.\right\}\cup \left\{\K\in \Z^{n}\mid |k_{n}|>2j_{1}L\right\}\]
and
\[\mathscr{K}_{\J; 3}=\left\{\K\in \Z^{n}\left| D(\hat{\K}, \J)<\rho, \frac{\hat{\K}}{j_{1}}\notin \bigcup_{y\in\overline{\mathscr{Y}}}{V_{y, \J}}, |k_{n}|\leq 2j_{1}L\right.\right\}.\]
For each $1\leq i\leq 3$, we let
\begin{equation}
N_{i}=\sum_{\substack{1\leq j_{1}\leq J\\0\leq j_{2}, ..., j_{R}\leq j_{1}\\ (\J/j_{1})\in \tilde{T}_{n}}}{\left(\prod_{r=1}^{R}{b_{j_{r}}}\right)\left| \sum_{q\leq Q}{q^{n}\sum_{\K\in \mathscr{K}_{\J; i}}{I(q; \J; \K)}}\right| }
\end{equation}
such that
\begin{equation}\label{Nsplit}
N^{(1; (1, ..., 1); n)}(Q, \delta)\ll N_{1}+N_{2}+N_{3}
\end{equation}
and proceed to bound each $N_{i}$ seperately.

\section{Bounds for $N_{1}, N_{2}$ and $N_{3}$}\label{secbounds}
\begin{lem}\label{countK}
For any $K>0$ we have that
\[\left\{\frac{\hat{\K}}{j_{1}}\mid \hat{\K}\in \Z^{n-1},  D(\hat{\K}, \J)<K\right\}\subseteq [-L-K, L+K]^{n-1},\]
where $L$ is defined as in \ref{boundedfy}.
\end{lem}

\begin{proof}
For $(\hat{\K}/j_{1})\in [-L, L]^{n-1}$ the inclusion is obvious, so let $(\hat{\K}/j_{1})\not\in [-L, L]^{n-1}$. We have
\begin{align*}
K>D(\hat{\K}, \J)&=\min_{y\in \overline{\mathscr{Y}}}\inf_{\z\in V_{y, \J}}\left|\frac{\hat{\K}}{j_{1}}-\z\right|\geq \min_{\z\in [-L, L]^{n-1}}\left|\frac{\hat{\K}}{j_{1}}-\z\right|\\
&\geq \min_{\textbf{z}\in [-L, L]^{n-1}}\left|\frac{|\hat{\K}|}{j_{1}}-|\textbf{z}|\right|=\frac{|\hat{\K}|}{j_{1}}-L,
\end{align*}
hence $K+L>|\hat{\K}/j_{1}|$ as desired.\\
\end{proof}

\paragraph{Case $\K\in \mathscr{K}_{\J; 2}$.}
Let
\[D_{1}(\hat{\K}, \J)=j_{1}D(\hat{\K}, \J)=\min_{y\in\overline{\mathscr{Y}}}\dist(\hat{\K}, j_{1}V_{y, \J}).\]
For a fixed $k_{n}\in \Z$, consider the integral
\[\int_{\R^{n-1}}{\omega(\hat{\X}, y)e\left(qj_{1}\left(\hat{F}_{y, \textbf{j}}(\hat{\X})-\frac{\hat{\K}\cdot \hat{\X}+k_{n}y}{j_{1}}\right)\right)\D\hat{\X}}.\]
with $(\hat{\K}, k_{n})\in \mathscr{K}_{\J; 2}$ and $D(\hat{\K}, \J)\geq \rho$. Let 
\[\varphi_{y, 1}(\hat{\X})=\frac{j_{1}\hat{f}_{1, y}(\hat{\X})+\cdots +j_{R}\hat{f}_{R, y}(\hat{\X})-\hat{\K}\cdot \hat{\X}-k_{n}y}{D_{1}(\hat{\K},\J)}\]
and $\lambda_{1} = qD_{1}(\hat{\K}, \J)$. Then by definition of $V_{y, \J}$
\[|\nabla\varphi_{y, 1}(\hat{\X})|=\frac{|j_{1}\nabla \hat{f}_{1,y}+\cdots+j_{R}\nabla\hat{f}_{R, y}(\hat{\X})-\hat{\K}|}{D_{1}(\hat{\K},\J)}\geq 1\]
for $\hat{\X}\in U_{y}$. Let $U_{y}^{+}\subseteq \R^{n-1}$ be an open set such that $\overline{U_{y}}\subseteq U_{y}^{+}\subseteq \mathscr{D}$, $V_{y, \J}^{+}=\nabla F_{\J}\left(U_{y}^{+}\right)\subseteq [-2L, 2L]^{n-1}$,
\begin{equation}\label{rhodistance}
\min_{\textbf{z}\in \overline{{U_{y}}}}<\frac{\dist(\del \mathscr{D}, \del U_{y})}{4}
\end{equation}
for any $\hat{\X}\in U_{y}^{+}$ and
\[\min_{\textbf{z}\in \overline{U_{y}}}{\max_{\t\in \tilde{T}_{n}}{\left|\nabla(t_{1}f_{1, y}+\cdots + t_{R}f_{R, y})(\X)-\nabla(t_{1}f_{1, y}+\cdots + t_{R}f_{R, y})(\textbf{z})\right|}}<\frac{\rho}{2}\]
for any $\hat{\X}\in U_{y}^{+}$. Then we have
\[\nabla\varphi_{y, 1}(\hat{\X})\geq \frac{1}{2}\]
for all $\hat{\X}\in U_{y}^{+}$.\\
Now assume that $|k_{n}|\leq D_{1}(\hat{\K}, \J)$ or $|k_{n}|\leq 2j_{1}L$.

\begin{lem}\label{phasederivbound}
Let $i_{1}, ..., i_{n-1}\in \Z_{\geq 0}$ with $\sum_{\mu=1}^{n-1}{i_{\mu}}\leq \ell$. Then for all $\hat{\X}\in U_{y}^{+}$ we have
\[\left|\frac{\del^{i_{1}+\cdots + i_{n-1}}\varphi_{y, 1}}{\del^{i_{1}}_{x_{1}}\cdots \del^{i_{n-1}}_{x_{n-1}}}(\hat{\X})\right|\ll 1,\]
where the implicit constant depends only on $(i_{1}, ..., i_{n-1})$, $\rho$, $\tau$, $\varepsilon_{1}$ and upper bounds for (the absolute values of) finitely many derivatives of $f_{r}$ on $\mathscr{D}\times \mathscr{Y}$ for $1\leq r\leq R$.
\end{lem}

\begin{proof}
Choose $C>0$ such that
\[\frac{1}{C}\max_{\substack{\t\in \tilde{T}_{n}\\ y\in \overline{\mathscr{Y}}\\\hat{\X}\in \overline{U_{y}}}}|\nabla(t_{1}\hat{f}_{1, y}+\cdots + t_{R}\hat{f}_{R, y})(\hat{\X})|<\frac{1}{2}.\]
and assume $j_{1}C\leq |\hat{\K}|$. Then we have
\[\left|\frac{\hat{\K}}{|\hat{\K}|}-\frac{j_{1}\z}{|\hat{\K}|}\right|\geq 1-\frac{j_{1}\z}{|\hat{\K}|}>\frac{1}{2}\]
for all $\textbf{z}\in \bigcup_{y\in \mathscr{Y}}V_{y, \J}$ and hence
\[\frac{1}{|\hat{\K}|}D_{1}(\hat{\K}, \J)=\min_{y\in \overline{\mathscr{Y}}}\dist\left(\frac{\hat{\K}}{|\hat{\K}|}, \frac{j_{1}V_{y, \J}}{|\hat{\K}|}\right)\geq \frac{1}{2}.\]
Therefore
\begin{align*}
|\varphi_{y, 1}(\hat{\X})|&\leq\left| \frac{\frac{j_{1}}{|\hat{\K}|}\hat{f}_{1, y}(\hat{\X})+\cdots +\frac{j_{R}}{|\hat{\K}|}\hat{f}_{R, y}(\hat{\X})-\frac{\hat{\K}}{|\hat{\K}|}\hat{\X}}{\frac{1}{|\hat{\K}|}D_{1}(\hat{\K}, \J)} \right|+\left|\frac{k_{n}y}{D_{1}(\hat{\K}, \J)}\right|\\
&\leq 2\left(\frac{j_{1}}{|\hat{\K}|}|\hat{f}_{1, y}(\hat{\X})|+\cdots + \frac{j_{R}}{|\hat{\K}|}|\hat{f}_{R, y}(\hat{\X})|+\left|\frac{\hat{\K}}{|\hat{\K}|}\right|\cdot |\hat{\X}|\right)+\frac{|k_{n}|}{D_{1}(\hat{\K},\J)}|y|.
\end{align*}
Since $|\hat{\K}|>j_{1}C$, we have $\frac{j_{r}}{|\hat{\K}|}\leq 1$ for $1\leq r\leq R$. Let $M_{r}>0$ be the bound for $|\hat{f}_{r, y}|$ established in \ref{boundedfy} on $U_{y}^{+}$ for $1\leq r\leq R$ and $S=M_{1}+\cdots +M_{R}$. Now by assumption we either have 
\[\frac{|k_{n}|}{D_{1}(\hat{\K}, \J)}\leq 1\qquad \mbox{or}\qquad \frac{|k_{n}|}{D_{1}(\hat{\K}, \J)}\leq \frac{2L}{\rho},\]
hence the final term is also bounded independently of $\J$.
Then
\[|\varphi_{y, 1}(\hat{\X})|\leq 2(S+\tau)+\varepsilon_{1}|\qquad\mbox{or}\qquad|\varphi_{y, 1}(\hat{\X})|\leq 2(S+\tau)+\frac{2L}{\rho}\varepsilon_{1}\]
as desired. If $|\hat{\K}|<Cj_{1}$ we immediately conclude
\[\varphi_{y, 1}(\hat{\X})|\leq\left| \frac{\hat{f}_{1, y}(\hat{\X})+\frac{j_{2}}{j_{1}}\hat{f}_{2, y}(\hat{\X})\cdots +\frac{j_{R}}{j_{1}}\hat{f}_{R, y}(\hat{\X})-\frac{\hat{\K}}{j_{1}}\hat{\X}}{\rho} \right|+\left|\frac{k_{n}y}{D_{1}(\hat{\K}, \J)}\right|\ll 1\]
for all $\hat{\X}\in U_{y}^{+}$ with the same argument for the final term as above. For the first partial derivatives consider
\[\left| \frac{\del \varphi_{y, 1}}{\del x_{i}}(\hat{\X})\right|=\left|\frac{j_{1}\frac{\del \hat{f}_{1, y}}{\del x_{i}}(\hat{\X})+\cdots + j_{R}\frac{\del \hat{f}_{R, y}}{\del x_{i}}(\hat{\X}) -k_{i}}{D_{1}(\hat{\K}, \J)}\right|\]
which can be treated with a similar argument. For higher partial derivatives the terms $\hat{\K}\cdot\hat{\X}$ vanish and the desired result follows easily with
\begin{align*}
\left| \frac{\del^{i_{1}+\cdots+i_{n-1}} \varphi_{y, 1}}{\del_{x_{1}}^{i_{1}}\cdots \del_{x_{n-1}}^{i_{n-1}}}(\hat{\X})\right|&=\left|\frac{j_{1}\frac{\del^{i_{1}+\cdots + i_{n-1}} \hat{f}_{1, y}}{\del_{x_{1}}^{i_{1}}\cdots \del_{x_{n-1}}^{i_{n-1}}}(\hat{\X})+\cdots + j_{R}\frac{\del^{i_{1}+\cdots + i_{n-1}} \hat{f}_{R, y}}{\del_{x_{1}}^{i_{1}}\cdots \del_{x_{n-1}}^{i_{n-1}}}(\hat{\X}) }{D_{1}(\hat{\K}, \J)}\right|\\
&\leq \frac{1}{\rho}\left(\left|\frac{\del^{i_{1}+\cdots + i_{n-1}} \hat{f}_{1, y}}{\del_{x_{1}}^{i_{1}}\cdots \del_{x_{n-1}}^{i_{n-1}}}(\hat{\X})\right|+\sum_{r=2}^{R}{\frac{j_{r}}{j_{1}}\left|\frac{\del^{i_{1}+ \cdots + i_{n-1}} \hat{f}_{r, y}}{\del_{x_{1}}^{i_{1}}\cdots \del_{x_{n-1}}^{i_{n-1}}}(\hat{\X})\right|}\right).
\end{align*}
\end{proof}
Therefore with Lemma \ref{nonstat} for $\varphi=\varphi_{y, 1}$ and $\lambda=\lambda_{1}$ as chosen above we have
\begin{equation}
\int_{\R^{n-1}}{\omega(\hat{\X}, y)e\left(qj_{1}\left(\hat{F}_{y, \textbf{j}}(\hat{\X})-\frac{\hat{\K}\cdot \hat{\X}+k_{n}y}{j_{1}}\right)\right)\D\hat{\X}}\ll(qD_{1}(\hat{\K},\J))^{-\ell+1}
\end{equation}
and hence
\begin{equation}\label{kbig}
I(q; \J; \K)\ll \int_{\mathscr{Y}}{(qD_{1}(\hat{\K}, \J))^{-\ell+1}\D y}\ll_{\varepsilon_{1}} (qD_{1}(\hat{\K}, \J))^{-\ell+1}.
\end{equation}
Now assume that $|k_{n}|> D_{1}(\hat{\K}, \J)$ and $|k_{n}|>2j_{1}L$. Consider
\[\int_{\R}{\omega(\hat{\X}, y)e\left(qj_{1}\left(\hat{F}_{y, \textbf{j}}(\hat{\X})-\frac{\hat{\K}\cdot \hat{\X}+k_{n}y}{j_{1}}\right)\right)\D y}.\]
Let $\lambda_{2}=qk_{n}$ and
\[\varphi_{\hat{\X}, 2}(y)=\frac{j_{1}f_{1}(\hat{\X}, y)+\cdots j_{R}f_{R}(\hat{\X}, y)-\hat{\K}\cdot\hat{\X}-k_{n}y}{k_{n}}.\]
Observe that
\begin{align*}
|\varphi_{\hat{\X}, 2}'(y)|&=\left|\frac{j_{1}\frac{\del f_{1}}{\del y}(\hat{\X}, y) +\cdots  j_{R}\frac{\del f_{R}}{\del y}(\hat{\X}, y) -k_{n}}{k_{n}}\right|\\
&\geq \left|1-\frac{\frac{\del f_{1}}{\del y}(\hat{\X}, y)+\frac{j_{2}}{j_{R}}\frac{\del f_{2}}{\del y}(\hat{\X}, y)+\cdots +\frac{j_{R}}{j_{1}}\frac{\del f_{R}}{\del y}(\hat{\X}, y)  }{2L}\right|\\
&\geq \frac{1}{2}.
\end{align*}
Consider further that
\begin{align*}
|\varphi_{\hat{\X}, 2}(y)|&=\left|\frac{j_{1}f_{1}(\hat{\X}, y)+\cdots j_{R}f_{R}(\hat{\X}, y)-\hat{\K}\cdot\hat{\X}-k_{n}y}{k_{n}}\right|\\
&\leq \frac{|f_{1}(\hat{\X}, y)|+\frac{j_{2}}{j_{1}}|f_{2}(\hat{\X}, y)|\cdots \frac{j_{R}}{j_{1}}|f_{R}(\hat{\X}, y)|+\frac{|\hat{\K}|}{j_{1}}\cdot|\hat{\X}|}{|k_{n}|}+|y|\\
&\leq \frac{|f_{1}(\hat{\X}, y)|+\frac{j_{2}}{j_{1}}|f_{2}(\hat{\X}, y)|\cdots \frac{j_{R}}{j_{1}}|f_{R}(\hat{\X}, y)|+(L+|k_{n}|)\cdot|\hat{\X}|}{|k_{n}|}+|y|\\
&\leq \frac{|f_{1}(\hat{\X}, y)|+\frac{j_{2}}{j_{1}}|f_{2}(\hat{\X}, y)|\cdots \frac{j_{R}}{j_{1}}|f_{R}(\hat{\X}, y)|+L|\hat{\X}|}{2L}+|\hat{\X}|+|y|,
\end{align*}
and
\begin{align*}
|\varphi_{\hat{\X}, 2}'(y)|&=\left|\frac{\frac{\del f_{1}}{\del y}(\hat{\X}, y)+ \frac{j_{2}}{j_{1}}\frac{\del f_{R}}{\del y}(\hat{\X}, y)\cdots + \frac{j_{R}}{j_{1}}\frac{\del f_{R}}{\del y}(\hat{\X}, y)-\frac{k_{n}}{j_{1}}y}{k_{n}}\right|\\
&\leq \frac{j_{1}}{|k_{n}|}\left|\frac{\del f_{1}}{\del y}(\hat{\X}, y)\right|+\cdots + \frac{j_{R}}{|k_{n}|}\left|\frac{\del f_{R}}{\del y}(\hat{\X},y)\right|+1,
\end{align*}
hence with analogous arguments as above we conclude that
\[\left|\frac{\del^{i}\varphi_{\hat{\X}, 2}}{\del y^{i}}(\hat{\X}, y)\right|\ll 1,\]
where the implicit constants only depend on $i, \rho, \tau, \varepsilon_{1}$ and upper bounds for (the absolute values of) finitely many derivatives of $f_{r}$ on $\mathscr{D}\times \mathscr{Y}$.
So Lemma \ref{nonstat} applies in the one dimensional case, hence
\begin{equation}
\int_{\R}{\omega(\hat{\X}, y)e\left(qj_{1}\left(\hat{F}_{y, \textbf{j}}(\hat{\X})-\frac{\hat{\K}\cdot \hat{\X}+k_{n}y}{j_{1}}\right)\right)\D y}\ll (qk_{n})^{-\ell +1}
\end{equation}
and therefore
\begin{equation}\label{knbig}
I(q; \J; \K)\ll \int_{U_{y}}{(qk_{n})^{-\ell+1}\D \hat{\X}}\ll_{\kappa} (qk_{n})^{-\ell+1}.
\end{equation}
Given these estimates we can split up the sum
\[\sum_{\K\in \mathscr{K}(\J; 2)}{I(q; \textbf{j}; \textbf{k})}=\sum_{\textbf{k}\in \mathscr{K}(\J; 2)}{\int_{\R^{n}}{\omega(\hat{\X}, y)e\left(qj_{1}\left(\hat{F}_{y, \textbf{j}}(\hat{\X})-\frac{\hat{\K}\cdot \hat{\X}+k_{n}y}{j_{1}}\right)\right)\D\hat{\X}\D y} }\]
as follows
\begin{align*}
\sum_{(\hat{\K}, k_{n})\in \mathscr{K}(\J; 2)}{I(q; \J; (\hat{\K}, k_{n}))}&=\sum_{\substack{D(\hat{\K}, \J)\geq \rho\\ D_{1}(\hat{\K}, \J)\geq |k_{n}|}}{I(q; \J, (\hat{\K}, k_{n}))}+\sum_{\substack{D(\hat{\K, \J})\geq \rho\\ 2j_{1}L\geq |k_{n}|> D_{1}(\hat{\K}, \J)\\}}{I(q; \J, (\hat{\K}, k_{n}))}\\
&\quad +\sum_{\substack{|k_{n}|> D_{1}(\hat{\K}, \J)\\|k_{n}|>2j_{1}L}}{I(q; \J; (\hat{\K}, k_{n}))}
\end{align*}
With (\ref{kbig}) and (\ref{knbig}) we obtain
\begin{align}\label{K2argument}
\sum_{\substack{D(\hat{\K}, \J)\geq \rho\\ D_{1}(\hat{\K}, \J)\geq |k_{n}|}}{I(q; \J, (\hat{\K}, k_{n}))}&\ll q^{-\ell+1}\sum_{\substack{D(\hat{\K}, \J)\geq \rho\\ D_{1}(\hat{\K}, \J)\geq |k_{n}|}}{D_{1}(\hat{\K}, \J)^{-\ell+1}}\\
&=q^{-\ell+1}\sum_{d=0}^{\infty}{\sum_{\substack{D_{1}(\hat{\K}, \J)\geq |k_{n}|\\ 2^{d}j_{1}\rho \leq D_{1}(\hat{\K}, \J)<2^{d+1}j_{1}\rho}}{D_{1}(\hat{\K}, \J)^{-\ell+1}}}\nonumber\\
&\leq q^{-\ell+1}\sum_{d=0}^{\infty}{\sum_{\substack{D_{1}(\hat{\K},  \J)\geq |k_{n}|\\ 2^{d}j_{1}\rho\leq D_{1}(\hat{\K},  \J)<2^{d+1}j_{1}\rho}}{\frac{1}{(2^{d}j_{1}\rho)^{\ell-1}}}}\nonumber\\
&\ll q^{-\ell+1}\sum_{d=0}^{\infty}{\frac{(j_{1}L+2^{d+1}j_{1}\rho)^{n}}{(2^{d}j_{1}\rho)^{\ell-1}}}\nonumber\\
&\ll_{L, n} q^{-\ell+1}\nonumber
\end{align}
and
\begin{align}
\sum_{\substack{D(\hat{\K}, \J)\geq \rho\\ 2j_{1}L\geq |k_{n}|>D_{1}(\hat{\K}, \J)}}{I(q; \J, (\hat{\K}, k_{n}))}&\ll q^{-\ell+1}\sum_{\substack{D(\hat{\K}, \J)\geq \rho\\ 2j_{1}L\geq |k_{n}|>D_{1}(\hat{\K}, \J)}}{D_{1}(\hat{\K}, \J)^{-\ell+1}}\\
&\leq q^{-\ell+1}\sum_{d=0}^{\infty}{\sum_{\substack{2j_{1}L\geq |k_{n}|\\ 2^{d}j_{1}\rho \leq D_{1}(\hat{\K}, \J)<2^{d+1}j_{1}\rho}}{D_{1}(\hat{\K}, \J)^{-\ell+1}}}\nonumber\\
&\leq q^{-\ell+1}\sum_{d=0}^{\infty}{\sum_{\substack{2j_{1}L\geq |k_{n}|\\ 2^{d}j_{1}\rho\leq D_{1}(\hat{\K},  \J)<2^{d+1}j_{1}\rho}}{\frac{1}{(2^{d}j_{1}\rho)^{\ell-1}}}}\nonumber\\
&\ll q^{-\ell+1}\sum_{d=0}^{\infty}{2j_{1}L\frac{(j_{1}L+2^{d+1}j_{1}\rho)^{n-1}}{(2^{d}j_{1}\rho)^{\ell-1}}}\nonumber\\
&\ll_{L, n} q^{-\ell+1}\nonumber
\end{align}
and
\begin{align}
\sum_{\substack{|k_{n}|>2j_{1}L\\ |k_{n}|> D_{1}(\hat{\K}, \J)}}{I(q; \K; \J)}&\ll q^{-\ell+1}\sum_{\substack{|k_{n}|>2j_{1}L\\ |k_{n}|>D_{1}(\hat{\K}, \J)}}{|k_{n}|^{-\ell+1}}\\
&\ll q^{-\ell+1} \sum_{\substack{|k_{n}|>2j_{1}L}}{\frac{(j_{1}L+|k_{n}|)^{n-1}}{|k_{n}|^{\ell-1}}}\nonumber\\
&\ll q^{-\ell+1} 2^{n-1}\sum_{\substack{|k_{n}|>2j_{1}L}}{|k_{n}|^{n-\ell}}\nonumber\\
&\ll q^{-\ell+1}\int_{2j_{1}L}^{\infty}{\frac{1}{t^{\ell-n}}\D t}\nonumber\\
&\ll q^{-\ell+1}j_{1}^{n-\ell+1}.\nonumber
\end{align}
Here the implicit constants only depend on $L, \rho,  n, \ell, \varepsilon_{1}, \kappa$ and upper bounds for (the absolute values of) finitely many derivatives of $\omega$ and $f_{r}$ for $1\leq r\leq R$ on $\mathscr{D}\times \mathscr{Y}$. Consequently we obtain
\begin{align}\label{case2}
N_{2}&\ll \sum_{\substack{1\leq j_{1}\leq J\\ 0\leq j_{2}, ..., j_{R}\leq j_{1}\\ \J/j_{1}\in \tilde{T}_{n}}}{\left(\prod_{r=1}^{R}{\left(\frac{1}{J}+\min\left(\delta, \frac{1}{j_{r}}\right)\right)}\right)j_{1}^{n-\ell+1}\sum_{q\leq Q}q^{n-\ell+1}}\\
&\ll \sum_{\substack{1\leq j_{1}\leq J\\ 0\leq j_{2}, ..., j_{R}\leq j_{1}}}{\left(\prod_{r=1}^{R}{\left(\frac{1}{J}+\min\left(\delta, \frac{1}{j_{r}}\right)\right)}\right)\int_{1}^{Q}{q^{n-\ell+1}\D q}}\nonumber\\
&\ll \sum_{\substack{1\leq j_{1}\leq J\\ 0\leq j_{2}, ..., j_{R}\leq j_{1}}}{\left(\prod_{r=1}^{R}{\left(\frac{1}{J}+\min\left(\delta, \frac{1}{j_{r}}\right)\right)}\right)\int_{1}^{Q}\frac{1}{q}\D q}\nonumber\\
&= \log Q \sum_{\substack{1\leq j_{1}\leq J\\ 0\leq j_{2}, ..., j_{R-1}\leq j_{1}}}\left(\left(\prod_{r=1}^{R}{\frac{1}{J}+\min\left(\delta, \frac{1}{j_{r}}\right)}\right) \sum_{0\leq j_{R}\leq j_{1}}{\left(\frac{1}{J}+\min\left(\delta, \frac{1}{j_{r}}\right)\right)}\right)\nonumber\\
&= \log Q \sum_{1\leq j_{1}\leq J}{\left(\left(\frac{1}{J}+\min\left(\delta, \frac{1}{j_{1}}\right)\right)\prod_{r=2}^{R}{\sum_{0\leq j_{r}\leq j_{1}}{\left(\frac{1}{J}+\min \left(\delta, \frac{1}{j_{r}}\right)\right)}}\right)}\nonumber\\
&\leq \log Q \prod_{r=1}^{R}{\sum_{0\leq j_{r}\leq J}{\left(\frac{1}{J}+\min\left(\delta, \frac{1}{j_{r}}\right)\right)}}\nonumber\\
&= \log Q \left(\sum_{0\leq j\leq J}{\left(\frac{1}{J}+\min\left(\delta, \frac{1}{j}\right)\right)}\right)^{R}\nonumber\\
&\ll \log Q (1+\log J)^{R}.\nonumber
\end{align} 

\paragraph{Case $\K\in \mathscr{K}_{\J; 3}$.} Let
$\lambda=qj_{1}$ and
\[\varphi_{y}(\hat{\X})=\hat{F}_{y, \J}(\hat{\X})-\frac{\hat{\K}}{j_{1}}\hat{\X}-\frac{k_{n}}{j_{1}}y.\]
For each $y$ we know that $\nabla\hat{F}_{y, \J}$ is a diffeomorphism on $\mathscr{D}$, hence for any fixed $\J$ and any $\hat{\K}$ with $(\hat{\K}/j_{1})\in \nabla \hat{F}_{y, \J}(\mathscr{D})$ we have a unique preimage
\[\xcrit=(\nabla \hat{F}_{y, \J})^{-1}(\hat{\K}/j_{1}).\]
This defines a critical point for $\varphi_{y}$, since
\[\nabla\varphi_{y}(\xcrit)=\nabla \hat{F}_{y, \J}(\xcrit)-\frac{\hat{\K}}{j_{1}}=\0.\]
Let $\mathscr{D}_{+}$ be an open set such that $\overline{\mathscr{D}}\subseteq \mathscr{D}_{+}\subseteq B_{3\tau/2}(\hat{\X}_{0})$.

\begin{lem}\label{nonstatgradientbound}
Let $\hat{\X}\in \mathscr{D}_{+}\backslash\{\xcrit\}$. Then 
\[\frac{|\hat{\X}-\xcrit|}{|\nabla \varphi_{y}(\hat{\X})|}\ll 1\]
where the implicit constant is independant of $\J$ and $\hat{\K}$.
\end{lem}

\begin{proof}
By definition of $\xcrit$ as a preimage of $\nabla \hat{F}_{y, \J}$ we have
\[\frac{|\hat{\X}-\xcrit|}{|\nabla \varphi_{y}(\hat{\X})|}= \frac{|\hat{\X}-\xcrit|}{|\nabla \hat{F}_{y, \J}(\hat{\X})-\frac{\hat{\K}}{j_{1}}|}=\frac{|\hat{\X}-\xcrit|}{|\nabla \hat{F}_{y, \J}(\hat{\X})-\nabla \hat{F}_{y, \J}(\xcrit)|}.\]
Now for any distinct $\hat{\X}, \hat{\z}\in \overline{\mathscr{D}_{+}}$ we know by Taylor's theorem that
\[\nabla \hat{F}_{y, \J}(\hat{\X})-\nabla \hat{F}_{y, \J}(\hat{\z})=H_{\hat{F}_{y, \J}}(\hat{\z})(\hat{\X}-\hat{\z})+O(|\hat{\X}-\hat{\z}|^{2}),\]
where the implicit constant does not depend on $\J$. Considering the eigenvalues of the invertible real symmetric matrix $H_{\hat{F}_{y, \J}}$ we have that
\[\lambda_{\min}|\hat{\X}-\hat{\z}|\ll |H_{\hat{F}_{y, \J}}(\hat{\z})(\hat{\X}-\hat{\z})|,\]
where $\lambda_{\min}$ is the minimum of the absolute values of the eigenvalues and the implicit constant depends only on $n$. We already showed that $|\det H_{\hat{F}_{y, \J}}|$ is bounded away from zero on $\overline{\mathscr{D}_{+}}$ hence by virtue of the eigenvalues being continuous in the coefficients of the matrix we find constants $C, \eta>0$ such that
\[|\hat{\X}-\hat{\z}|\leq C|\nabla \hat{F}_{y, \J}(\hat{\X})-\nabla \hat{F}_{y, \J}(\hat{\z})|\]
for all $|\hat{\X}-\hat{\z}|\leq \eta$. Here $C, \eta$ are independend of $\J$. Now in particular choosing $\hat{\z}=\xcrit$ yields the desired result.
\end{proof}

Following similar arguments as in \ref{phasederivbound} we find that for $i_{1}, ..., i_{n-1}\in \Z_{\geq 0}$ with $\sum_{\mu=1}^{n-1}{i_{\mu}}\leq \ell$. Then for all $\hat{\X}\in \mathscr{D}_{+}$ we have
\[\left|\frac{\del^{i_{1}+\cdots + i_{n-1}}\varphi_{y}}{\del^{i_{1}}_{x_{1}}\cdots \del^{i_{n-1}}_{x_{n-1}}}(\hat{\X})\right|\ll 1,\]
where the implicit constant depends only on $(i_{1}, ..., i_{n-1})$, $\rho$, $\tau$, $\varepsilon_{1}$ and upper bounds for (the absolute values of) finitely many derivatives of $f_{r}$ on $\mathscr{D}_{+}\times \mathscr{Y}$ for $1\leq r\leq R$. Notice that scaling with $D_{1}(\hat{\K}, \J)$ is unnecessary here since for $\K\in \mathscr{K}_{\J; 3}$ this distance is bounded from above by $\rho$ and that the bounds are all independant of $\J$ and $\K$. By definition $H_{\varphi_{y}}=H_{\hat{F}_{y, \J}}$, hence we can apply \ref{stat} for $\varphi=\varphi_{y},$ and $\lambda$ as above together with (\ref{determinantbound}) to obtain
\begin{align*}
I(q; \J, \K)&\ll \int_{\mathscr{Y}}{\abs{\det H_{\hat{F}_{y, \J}}(\xcrit)}^{-\frac{1}{2}}(qj_{1})^{-\frac{n-1}{2}-1}\D y}\\
&\ll_{c_{1}, \varepsilon_{1}} (qj_{1})^{-\frac{n}{2}-\frac{1}{2}},
\end{align*}
where the implicit constant only depends on $\ell, n, \varepsilon_{1}$, upper bounds for the absolute values of finitely many derivatives of $\omega$ and $\varphi_{y}$ on $\mathscr{D}_{+}$, an upper bound for $|\X-\textbf{v}_{0}|/|\nabla\varphi_{y}(\textbf{v}_{0})|$ on $\mathscr{D}_{+}$, and a lower bound for $\det H_{\hat{F}_{y, \J}}(\xcrit)$, all of which are independant of $\J$ and $\K$. We obtain
\begin{align*}
\sum_{\substack{(\hat{\K}, k_{n})\in \mathscr{K}(\J, 3)}}{I(q; \J, (\hat{\K}, k_{n}))}&\ll \sum_{\substack{(\hat{\K}, k_{n})\in \mathscr{K}(\J, 3)\\ }}{(qj_{1})^{-\frac{n}{2}-\frac{1}{2}}}\ll 2j_{1}L\sum_{\substack{\hat{\K}\in \Z^{n-1}\\D(\hat{\K}, \J)< \rho}}{(qj{_1})^{-\frac{n}{2}-\frac{1}{2}}}\\
&\ll 2j_{1}\sum_{\substack{\hat{\K}\in \Z^{n-1}\\ (\hat{\K}/j_{1})\in [-L-\rho, L+\rho]^{n-1}
}}{(qj_{1})^{-\frac{n}{2}-\frac{1}{2}}}\\
&\ll q^{-\frac{n}{2}-\frac{1}{2}} j_{1}^{\frac{n}{2}-\frac{1}{2}},
\end{align*}
where the implicit constant only depends additionally on $L, \rho$ and $n$. Arguing similarly to (\ref{case2}) we obtain
\begin{align}\label{case3}
N_{3}&\ll \sum_{\substack{1\leq j_{1}\leq J\\ 0\leq j_{2}, ..., j_{R}\leq j_{1}\\ \J/j_{1}\in \tilde{T}_{n}}}{\left(\prod_{r=1}^{R}{\left(\frac{1}{J}+\min\left(\delta, \frac{1}{j_{r}}\right)\right)}\right)j_{1}^{\frac{n}{2}-\frac{1}{2}}\sum_{q\leq Q}q^{\frac{n}{2}-\frac{1}{2}}}\\
&\ll \sum_{\substack{1\leq j_{1}\leq J\\ 0\leq j_{2}, ..., j_{R}\leq j_{1}}}{\left(\prod_{r=1}^{R}{\left(\frac{1}{J}+\min\left(\delta, \frac{1}{j_{r}}\right)\right)}\right)J^{\frac{n}{2}-\frac{1}{2}}Q^{\frac{n}{2}+\frac{1}{2}}}\nonumber\\
&\ll J^{\frac{n}{2}-\frac{1}{2}}Q^{\frac{n}{2}+\frac{1}{2}}(1+\log J)^{R}.\nonumber
\end{align} 

\paragraph{Case $\K\in \mathscr{K}_{\J; 1}$.} Choose $\varphi_{y}$ and $\lambda$ as in the previous case, such that we still have
\[\nabla\varphi_{y}(\xcrit)= \nabla\hat{F}_{y, \J}(\xcrit)-\frac{\hat{\K}}{j_{1}}\]
and \ref{nonstatgradientbound} still applies. The signature $\sigma$ of the matrix $H_{\varphi_{y}}(\hat{\X}_{\J; \K})=H_{\hat{F}_{y, \J}}(\hat{\X}_{\J; \K})$ is constant for all $\J, \K$ and $y$ in consideration, since the determinant is bounded away from zero, the eigenvalues of a matrix depend continuously on its coefficients and the coefficients depend continously on $y$. Applying Lemma \ref{stat} again yields
\begin{equation}\label{Integraly}
I(q;\J;\K)\ll \left|\int_{\R}{(qj_{1})^{-\frac{n}{2}+\frac{1}{2}}\frac{\omega(\xcrit, y)}{|\det H_{\hat{F}_{y, \J}}(\xcrit)|^{\frac{1}{2}}}e\left(qj_{1}\varphi_{y}(\xcrit)+\frac{\sigma}{8}\right)\D y}\right|+(qj_{1})^{-\frac{n}{2}-\frac{1}{2}}.
\end{equation}
Since all of $\omega, \varphi_{y}$ and $\xcrit$ depend on $y$, we delay evaluating the integral for now and consider the terms
\begin{align*}
&N_{1, y, \J}\\
&=\left| \sum_{\K\in \mathscr{K}_{\J; 1}}{\sum_{q\leq Q}{q^{n}\left((qj_{1})^{-\frac{n}{2}+\frac{1}{2}}\frac{\omega(\xcrit, y)}{|\det H_{\hat{F}_{y, \J}}(\xcrit)|^{\frac{1}{2}}}e\left(qj_{1}\varphi_{y}(\xcrit)+\frac{\sigma}{8}\right)\right) }}\right|\\
&+\bigO{\sum_{\K\in \mathscr{K}_{\J; 1}}{\sum_{q\leq Q}{q^{n}(qj_{1})^{-\frac{n}{2}-\frac{1}{2}}}}}\\
&=\left| \sum_{\K\in \mathscr{K}_{\J; 1}}{\sum_{q\leq Q}{q^{n}\left((qj_{1})^{-\frac{n}{2}+\frac{1}{2}}\frac{\omega(\xcrit, y)}{|\det H_{\hat{F}_{y, \J}}(\xcrit)|^{\frac{1}{2}}}e\left(qj_{1}\varphi_{y}(\xcrit)+\frac{\sigma}{8}\right)\right) }}\right|\\
&+\bigO{j_{1}^{\frac{n}{2}-\frac{1}{2}}Q^{\frac{n}{2}+\frac{1}{2}}}
\end{align*}
and
\[N_{1, y}=\sum_{\substack{1\leq j_{1}\leq J\\0\leq j_{2}, ..., j_{R}\leq j_{1}\\ (\J/j_{1})\in T_{n}}}{\left(\prod_{r=1}^{R}{b_{j_{r}}}\right)N_{1, y, \J}}.\]
We start with the inner most sum
\begin{align}\label{qsum}
&\left|\sum_{q\leq Q}{q^{n}\left((qj_{1})^{-\frac{n}{2}+\frac{1}{2}}\frac{\omega(\hat{\X}_{\J; \hat{\K}}, y)}{|\det H_{\hat{F}_{y, \J}}(\hat{\X}_{\J; \hat{\K}})|^{\frac{1}{2}}}e\left(qj_{1}\varphi_{y}(\hat{\X}_{\J; \hat{\K}})+\frac{\sigma}{8}\right)\right)}\right|\\
&\quad \leq \frac{\omega(\xcrit, y)}{|\det H_{\hat{F}_{y, \J}}(\xcrit)|^{\frac{1}{2}}} j_{1}^{-\frac{n}{2}+\frac{1}{2}}\left|\sum_{q\leq Q}{\left(q^{\frac{n}{2}+\frac{1}{2}} e\left(qj_{1}\varphi_{y}(\xcrit) \right)\right)}\right|\nonumber\\
&\quad\ll \frac{\omega(\xcrit, y)}{|\det H_{\hat{F}_{y, \J}}(\xcrit)|^{\frac{1}{2}}} j_{1}^{-\frac{n}{2}+\frac{1}{2}} \left|\sum_{q\leq Q}{\left(q^{\frac{n}{2}+\frac{1}{2}}e(qj_{1}\varphi_{y}(\xcrit))\right)}\right|.\nonumber
\end{align}
The remaining sum over $q$ can be dealt with by means of partial summation
\begin{align*}\sum_{q\leq Q}{q^{\frac{n}{2}+\frac{1}{2}}e(qj_{1}\varphi_{y}(\xcrit))}&=Q^{\frac{n}{2}+\frac{1}{2}}\sum_{q\leq Q}{e(qj_{1}\varphi_{y}(\xcrit))}\\
&-\int_{1}^{Q}{\sum_{q\leq \xi}{e(qj_{1}\varphi_{y}(\xcrit))}\left(\frac{n}{2}+\frac{1}{2}\right)\xi^{\frac{n}{2}-\frac{1}{2}}\D \xi}
\end{align*}
Note that we have the bound
\[\left|\sum_{q\leq Q}{e(qj_{1}\varphi_{y}(\xcrit))}\right|\ll\min\{Q, ||j_{1}\varphi_{y}(\xcrit)||^{-1}\},\]
so we distinguish two cases. First if $||j_{1}\varphi_{y}(\xcrit)||\geq Q^{-1}$ we obtain
\begin{align}\label{qsumcase1}
\sum_{q\leq Q}{q^{\frac{n}{2}+\frac{1}{2}}e(qj_{1}\varphi_{y}(\xcrit))}&\ll \frac{Q^{\frac{n}{2}+\frac{1}{2}}}{||j_{1}\varphi_{y}(\xcrit)||}+\frac{1}{||j_{1}\varphi_{y}(\xcrit)||}\int_{1}^{Q}{\left(\frac{n}{2}+\frac{1}{2}\right)\xi^{\frac{n}{2}-\frac{1}{2}}\D \xi}\\
&\ll \frac{Q^{\frac{n}{2}+\frac{1}{2}}}{||j_{1}\varphi_{y}(\xcrit)||}.\nonumber
\end{align}
On the other hand if $||j_{1}\varphi_{y}(\xcrit)||< Q^{-1}$we have
\begin{align}\label{qsumcase2}
\sum_{q\leq Q}{q^{\frac{n}{2}+\frac{1}{2}}e(qj_{1}\varphi_{y}(\xcrit))}&\ll Q^{\frac{n}{2}+\frac{3}{2}}+\left(\frac{n}{2}+\frac{1}{2}\right)\int_{1}^{Q}{\xi^{\frac{n}{2}+\frac{1}{2}}\D \xi}\\
&\ll Q^{\frac{n}{2}+\frac{3}{2}}.\nonumber
\end{align}
Hence we obtain
\begin{align}\label{N1y}
N_{1, y} &\ll \sum_{\substack{1\leq j_{1}\leq J\\ 0\leq j_{2}, ..., j_{R}\leq j_{1}}}{\left(\prod_{r=1}^{R}{b_{j_{r}}}\right) \sum_{\substack{\K\in \mathscr{K}_{\J; 1} \\ ||j_{1}\varphi_{y}(\xcrit)||\geq Q^{-1}}}{\omega(\xcrit)j_{1}^{-\frac{n}{2}+\frac{1}{2}}Q^{\frac{n}{2}+\frac{1}{2}}||j_{1}\varphi_{y}(\xcrit)||^{-1}} }\\
&+\sum_{\substack{1\leq j_{1}\leq J\\ 0\leq j_{2}, ..., j_{R}\leq j_{1}}}{\left(\prod_{r=1}^{R}{b_{j_{r}}}\right) \sum_{\substack{\K\in \mathscr{K}_{\J; 1} \\ ||j_{1}\varphi_{y}(\xcrit)||< Q^{-1}}}{\omega(\xcrit)j_{1}^{-\frac{n}{2}+\frac{1}{2}}Q^{\frac{n}{2}+\frac{3}{2}}} } \nonumber\\
&+ \sum_{\substack{1\leq j_{1}\leq J\\ 0\leq j_{2}, ..., j_{R}\leq j_{1}}}{\left(\prod_{r=1}^{R}{b_{j_{r}}}\right) \sum_{\K\in \mathscr{K}_{\J; 1}}{j_{1}^{-\frac{n}{2}-\frac{1}{2}}Q^{\frac{n}{2}+\frac{1}{2}}} } \nonumber\\
&\ll_{L} Q^{\frac{n}{2}+\frac{1}{2}}\sum_{\substack{1\leq j_{1}\leq J\\ 0\leq j_{2}, ..., j_{R}\leq j_{1}}}{\left(\prod_{r=1}^{R}{b_{j_{r}}}\right) \sum_{\substack{\K \in\mathscr{K}_{\J; 1}  \\ ||j_{1}\varphi_{y}(\xcrit)||\geq Q^{-1}}}{\omega(\xcrit)j_{1}^{-\frac{n}{2}+\frac{1}{2}}||j_{1}\varphi_{y}(\xcrit)||^{-1}} } \nonumber\\
&+Q^{\frac{n}{2}+\frac{3}{2}}\sum_{\substack{1\leq j_{1}\leq J\\ 0\leq j_{2}, ..., j_{R}\leq j_{1}}}{\left(\prod_{r=1}^{R}{b_{j_{r}}}\right) \sum_{\substack{\K\in\mathscr{K}_{\J; 1}  \\ ||j_{1}\varphi_{y}(\xcrit)||< Q^{-1}}}{\omega(\xcrit)j_{1}^{-\frac{n}{2}+\frac{1}{2}}} } \nonumber\\
&+ \sum_{\substack{1\leq j_{1}\leq J\\ 0\leq j_{2}, ..., j_{R}\leq j_{1}}}{\left(\prod_{r=1}^{R}{b_{j_{r}}}\right)j_{1}^{\frac{n}{2}-\frac{1}{2}}Q^{\frac{n}{2}+\frac{1}{2}}} \nonumber
\end{align}
The last term can be bounded similarly to (\ref{case2}) and (\ref{case3}). We have the following essential result to be proven in Section \ref{secess}.

\begin{prop}\label{essential}
Let $T>0$ and $J_{2}, ..., J_{R}\in [1, J]$. Then with the notations of this section and for all $y\in \overline{\mathscr{Y}}$, we have
\[\sum_{\substack{1\leq j_{1}\leq J\\0\leq j_{r}\leq \min\{J_{r}, j_{1}\} \\2\leq r\leq R}}{\sum_{\substack{\K\in \Z^{n}\\ ||j_{1}\varphi_{y}(\xcrit)||<T^{-1}}}{\omega(\xcrit)}}\ll \left(\prod_{r=2}^{R}{J_{r}}\right)(J^{n+1}T^{-1}+J^{n}\mathscr{E}_{n-1}(J)),\]
where
\[\mathscr{E}_{n-1}(J)=\mathscr{E}_{n-1}^{(\idl{c}_{1}'; \idl{c}_{2}')}=\left\{\begin{array}{*{2}{c}} \exp(\idl{c}'_{1}\sqrt{\log J})& \mbox{if $n=3$}\\ (\log J)^{\idl{c}'_{2}}& \mbox{if $n\geq 4$}\end{array}\right.\]
for some positive constants $\idl{c}'_{1}$ and $\idl{c}'_{2}$. Here the implicit constant as well as $\idl{c}'_{1}$ and $\idl{c}'_{2}$ only depend on $n, R$, $c_{1}$ and $c_{2}$ in (\ref{determinantbound}), $\rho$ in (\ref{rho}), $\rho'$ in (\ref{rho'}) and upper bounds for (the absolute value) of finitely many derivatives of $\omega$ and $f_{r}$ for $2\leq r\leq R$ on $\mathscr{D}_{+}\times \mathscr{Y} $. In particular, they are independant of $T$, $J_{2}, ..., J_{R}$  and $y$.
\end{prop}
Recall that $b_{j}\ll \frac{1}{j}$ for $1\leq j\leq J$. Hence with $\idl{I}_{0}=\{0\}$ and $\idl{I}_{s}=[2^{s-1}, 2^{s}]$ it follows that
\begin{align}
&\sum_{\substack{1\leq j_{1}\leq J\\0\leq j_{2}, ..., j_{R}\leq j_{1}}}{\left(\prod_{r=1}^{R}{b_{j_{r}}}\right)j_{1}^{-\frac{n}{2}+\frac{1}{2}}\sum_{\substack{\K\in \mathscr{K}_{\J; 1}\\||j_{1}\varphi_{y}(\xcrit)||<T^{-1}}}{\omega(\xcrit)}}\\
&\ll \sum_{0\leq s_{2}, ..., s_{R}\leq \frac{\log J}{\log 2}+1}{\left(\prod_{r=2}^{R}{2^{-s_{r}}}\right)\sum_{\substack{1\leq j_{1}\leq J\\ j_{r}\in \idl{I}_{s_{r}}\cap [0, j_{1}]\\ 2\leq r\leq R}}{j_{1}^{-\frac{n}{2}-\frac{1}{2}}}\sum_{\substack{\K\in \mathscr{K}_{\J; 1}\\||j_{1}\varphi_{y}(\xcrit)||<T^{-1}}}{\omega(\xcrit) } }\nonumber\\
&\ll \sum_{0\leq s_{2}, ..., s_{R}\leq \frac{\log J}{\log 2}+1}{\left(\prod_{r=2}^{R}{2^{-s_{r}}}\right)\sum_{\substack{1\leq j_{1}\leq J\\ 0\leq j_{r}\leq \min\{2^{s_{r}}, j_{1}\}\\ 2\leq r\leq R}}{j_{1}^{-\frac{n}{2}-\frac{1}{2}}}\sum_{\substack{\K\in \mathscr{K}_{\J; 1}\\||j_{1}\varphi_{y}(\xcrit)||<T^{-1}}}{\omega(\xcrit) } }.\nonumber
\end{align}
Now using partial summation and Proposition \ref{essential}, we find that for all $s_{r}$ in consideration
\begin{align}
&\sum_{\substack{1\leq j_{1}\leq J\\ 0\leq j_{r}\leq \min\{2^{s_{r}}, j_{1}\}\\ 2\leq r\leq R}}{j_{1}^{-\frac{n}{2}-\frac{1}{2}}}\sum_{\substack{\K\in \mathscr{K}_{\J; 1}\\||j_{1}\varphi_{y}(\xcrit)||<T^{-1}}}{\omega(\xcrit)}\\
&\qquad \ll J^{-\frac{n}{2}-\frac{1}{2}}\left(\prod_{r=2}^{R}{2^{s_{r}}}\right)(J^{n+1}T^{-1}+J^{n}\mathscr{E}_{n-1}(J)).\nonumber
\end{align}
Therefore
\begin{align}\label{N1parts}
&\sum_{\substack{1\leq j_{1}\leq J\\0\leq j_{2}, ..., j_{R}\leq j_{1}}}{\left(\prod_{r=1}^{R}{b_{j_{r}}}\right)j_{1}^{-\frac{n}{2}+\frac{1}{2}}\sum_{\substack{\K\in \mathscr{K}_{\J; 1}\\||j_{1}\varphi_{y}(\xcrit)||<T^{-1}}}{\omega(\xcrit)}}\\
&\ll (1+\log J)^{R}J^{-\frac{n}{2}-\frac{1}{2}}(T^{-1}J^{n+1}+J^{n}\mathscr{E}_{n-1}(J)).\nonumber
\end{align}
Now the second term in (\ref{N1y}) can be estimated by taking $T=Q$ in (\ref{N1parts}). For the first sum in (\ref{N1y}) we split the interval $[Q^{-1}, 1/2]$ into dyadic intervals. We may assume $Q\geq 2$, i.e. $Q^{-1}\leq 1/2$, and conclude 
\begin{align}\label{N1yfirst}
&\sum_{\substack{1\leq j_{1}\leq J\\ 0\leq j_{2}, ..., j_{R}\leq j_{1}}}{\left(\prod_{r=1}^{R}{b_{j_{r}}}\right) \sum_{\substack{\K \in\mathscr{K}_{\J; 1}  \\ ||j_{1}\varphi_{y}(\xcrit)||\geq Q^{-1}}}{\omega(\xcrit)j_{1}^{-\frac{n}{2}+\frac{1}{2}}||j_{1}\varphi_{y}(\xcrit)||^{-1}} }\\
&\qquad\leq \sum_{1\leq i\leq \frac{\log Q}{\log 2}+1}{Q2^{1-i}\sum_{\substack{1\leq j_{1}\leq J\\ 0\leq j_{2}, ..., j_{R}\leq j_{1}}}{\left(\prod_{r=1}^{R}{b_{j_{r}}}\right) \sum_{\substack{\K \in\mathscr{K}_{\J; 1}  \\ \frac{2^{i-1}}{Q}\leq ||j_{1}\varphi_{y}(\xcrit)||\leq \frac{2^{i}}{Q}}}{\omega(\xcrit)j_{1}^{-\frac{n}{2}+\frac{1}{2}}} }}\nonumber\\
&\qquad \ll  \sum_{1\leq i\leq\frac{\log Q}{\log 2}+1}{Q2^{1-i}(1+\log J)^{R}J^{-\frac{n}{2}-\frac{1}{2}}(2^{i}Q^{-1}J^{n+1}+J^{n}\mathscr{E}_{n-1}(J))} \nonumber\\
&\qquad \ll (1+\log J)^{R}((\log Q)J^{\frac{n}{2}+\frac{1}{2}}+QJ^{\frac{n}{2}-\frac{1}{2}}\mathscr{E}_{n-1}(J))\nonumber
\end{align}
using (\ref{N1parts}) again. Combining (\ref{N1y}), (\ref{N1parts}) and (\ref{N1yfirst}), we obtain
\begin{align}\label{N1yfinal}
N_{1, y}&\ll Q^{\frac{n}{2}+\frac{1}{2}} (1+\log J)^{R}((\log Q)J^{\frac{n}{2}+\frac{1}{2}}+QJ^{\frac{n}{2}-\frac{1}{2}}\mathscr{E}_{n-1}(J))\\
&+ Q^{\frac{n}{2}+\frac{3}{2}} (1+\log J)^{R}J^{-\frac{n}{2}-\frac{1}{2}}(Q^{-1}J^{n+1}+J^{n}\mathscr{E}_{n-1}(J))\nonumber\\
&+ (1+\log J)^{R}Q^{\frac{n}{2}+\frac{1}{2}}J^{\frac{n}{2}-\frac{1}{2}}\nonumber\\
&\ll (1+\log J)^{R}((\log Q)Q^{\frac{n}{2}+\frac{1}{2}}J^{\frac{n}{2}+\frac{1}{2}}+Q^{\frac{n}{2}+\frac{3}{2}}J^{\frac{n}{2}-\frac{1}{2}}\mathscr{E}_{n-1}(J)).\nonumber
\end{align}
Consequently we have
\begin{equation}\label{case1}
N_{1}\ll_{\varepsilon_{1}} (1+\log J)^{R}((\log Q)Q^{\frac{n}{2}+\frac{1}{2}}J^{\frac{n}{2}+\frac{1}{2}}+Q^{\frac{n}{2}+\frac{3}{2}}J^{\frac{n}{2}-\frac{1}{2}}\mathscr{E}_{n-1}(J)).
\end{equation}

\section{Proof of Proposition \ref{essential}}\label{secess}
Recall that we defined functions
\[\hat{f}_{j, y}\colon \R^{n-1}\rightarrow \R, \hat{\X}\mapsto f_{j}(\hat{\X}, y)\]
for $y\in \overline{\mathscr{Y}}$, $\varepsilon_{1}$ as in (\ref{determinantbound}), and
\[\hat{F}_{y, \J}=\hat{f}_{1, y}+\frac{j_{2}}{j_{1}}\hat{f}_{2, y}+\cdots + \frac{j_{R}}{j_{1}}\hat{f}_{R, y}\]
for $\J\in \R^{R}\setminus\{0\}$. For a non-negative weight function $\omega \in \mathcal{C}_{c}^{\infty}(\R^{n})$ we defined
\[U_{y}=\{\hat{\X}\in \R^{n-1}\mid \omega(\hat{\X}, y)\neq 0\}\]
and $V_{y, \J}=\nabla \hat{F}_{y, \J}$. Note that $\nabla\hat{F}_{y, \J}$ is a diffeomorphism on $U_{y}$. Now let $\omega_{\J}^{\ast}=\omega \circ (\nabla \hat{F}_{y, \J})^{-1}$, and for $T\geq 2$ and $J_{2}, ..., J_{R}\in [1, J]$, define
\begin{align}\label{newcp}
\mathscr{M}(J, T^{-1})&=\sum_{\substack{1\leq j_{1}\leq J\\ 0\leq j_{r}\leq \min\{J_{r}, j_{1}\}\\ 2\leq r\leq R}}{\sum_{\substack{\a\in \mathscr{K}_{\J; 1}\\||j_{1}\varphi_{y}(\hat{\X}_{\J; \hat{\a}})||\leq T^{-1}}}{\omega_{\J}^{\ast}\left(\frac{\hat{\a}}{j_{1}}\right)}}\\
&=\sum_{\substack{1\leq j_{1}\leq J\\ 0\leq j_{r}\leq \min\{J_{r}, j_{1}\}\\ 2\leq r\leq R}}{\sum_{\substack{\a\in \Z^{n}\\ |a_{n}|\leq 2j_{1}L\\||j_{1}\varphi_{y}(\hat{\X}_{\J; \hat{\a}})||\leq T^{-1}}}{\omega_{\J}^{\ast}\left(\frac{\hat{\a}}{j_{1}}\right)}}.\nonumber
\end{align}
Note that \ref{essential} for $0<T<2$ immediately follows from the case $T=2$. We consider the Fej\'{e}r kernel 
\begin{equation}\label{fejer}
\mathscr{F}_{D}(\theta)=D^{-2}\left| \sum_{d=1}^{D}{e(d\theta)}\right|^{2}=\left(\frac{\sin(\pi D\theta)}{D\sin (\pi\theta)}\right)^{2}=\sum_{d=-D}^{D}{\frac{D-|d|}{D^{2}}e(d\theta)}
\end{equation}
for $D=\subgauss{T/2}$. Let $\theta\in \R$ with $0< ||\theta||\leq T^{-1}$, then by the concave property of the sine function on $[0, \pi/2]$ we have
\[\left(\frac{\sin (D\pi \theta)}{D\sin(\pi\theta)}\right)^{2}\geq \left(\frac{2\pi^{-1}D\pi||\theta||}{D\pi ||\theta||}\right)\geq \frac{4}{\pi^{2}}.\]
Therefore, it follows that
\begin{equation}\label{fejerindicator}
\chi_{T^{-1}}(\theta)\leq \frac{\pi^{2}}{4} \mathscr{F}_{D}(\theta),
\end{equation}
with $\chi_{T^{-1}}$ as in (\ref{indi}). Combining (\ref{newcp}), (\ref{fejer}) and (\ref{fejerindicator}) we obtain
\begin{equation}\label{newsetup}
\mathscr{M}(J, T^{-1})\leq \frac{\pi^{2}}{4}\sum_{\substack{1\leq j_{1}\leq J\\ 0\leq j_{r}\leq \min\{J_{r}, j_{1}\}\\ 2\leq r\leq R}}{\sum_{\substack{\a\in \Z^{n}\\ |a_{n}|\leq 2j_{1}L}}{\sum_{d=-D}^{D}{\frac{D-|d|}{D^{2}} \omega_{\J}^{\ast}\left(\frac{\hat{\a}}{j_{1}}\right)e(dj_{1}\varphi_{y}(\hat{\X}_{\J; \hat{\a}})) }}}.
\end{equation}
By definition $\omega_{\J}^{\ast}$ vanishes outside of $\bigcup_{y\in\overline{\mathscr{Y}}}{V_{y, \J}}\subseteq [-L, L]^{n}$, hence the contribution of terms with $d=0$ in (\ref{newsetup}) is
\begin{equation}
\frac{\pi^{2}}{4D}\sum_{\substack{1\leq j_{1}\leq J\\ 0\leq j_{r}\leq \min\{J_{r}, j_{1}\}\\ 2\leq r\leq R}}{\sum_{\substack{\a\in \Z^{n}\\ |a_{n}|\leq 2j_{1}L}}{\omega_{\J}^{\ast}\left(\frac{\hat{\a}}{j_{1}}\right)}}\ll \frac{1}{D}\left(\prod_{r=2}^{R}{J_{r}}\right)\sum_{1\leq j_{1}\leq J}{j_{1}^{n}}\ll \frac{J^{n+1}}{D}\left(\prod_{r=2}^{R}{J_{r}}\right),
\end{equation}
where the implicit constants only depend on $n$ and $L$. Let $\hat{F}^{\ast}_{y, \J}$ be the Legendre transform of $\hat{F}_{y, \J}$. Then with (\ref{legendrepoint}) and since $\xcrita=(\nabla \hat{F}_{y, \J})^{-1}(\hat{a}/j_{1})$ we have
\[\hat{F}^{\ast}_{y, \J}\left(\frac{\hat{\a}}{j_{1}}\right)=\xcrita \cdot \frac{\hat{\a}}{j_{1}}-\hat{F}_{y, \J}(\xcrita)=-\varphi_{y}(\xcrita)-\frac{a_{n}y}{j_{1}}.\]
Now we can rewrite
\begin{align}
&\sum_{\substack{\a\in \Z^{n}\\ |a_{n}|\leq 2j_{1}L}}{\sum_{d=-D}^{D}{\frac{D-|d|}{D^{2}} \omega_{\J}^{\ast}\left(\frac{\hat{\a}}{j_{1}}\right)e(dj_{1}\varphi_{y}(\hat{\X}_{\J; \hat{\a}})) }}\\
&\qquad=\sum_{\substack{\a\in \Z^{n}\\ |a_{n}|\leq 2j_{1}L}}{\sum_{d=-D}^{D}{\frac{D-|d|}{D^{2}} \omega_{\J}^{\ast}\left(\frac{\hat{\a}}{j_{1}}\right)e\left(-dj_{1}\left(\frac{a_{n}y}{j_{1}}+\hat{F}^{\ast}_{y, \J}\left(\frac{\hat{\a}}{j_{1}}\right)\right)\right) }}\nonumber \\
&\qquad=\sum_{\substack{\a\in \Z^{n}\\ |a_{n}|\leq 2j_{1}L}}{\sum_{d=-D}^{D}{\frac{D-|d|}{D^{2}} \omega_{\J}^{\ast}\left(\frac{\hat{\a}}{j_{1}}\right)e(-da_{n}y)e\left(-dj_{1}\hat{F}^{\ast}_{y, \J}\left(\frac{\hat{\a}}{j_{1}}\right)\right) }}\nonumber\\
&\qquad=\sum_{|a_{n}|\leq 2j_{1}L}{\sum_{d=-D}^{D}{\frac{D-|d|}{D^{2}} e(-da_{n}y)\sum_{\hat{\a}\in \Z^{n-1}}{\omega_{\J}^{\ast}\left(\frac{\hat{\a}}{j_{1}}\right)e\left(-dj_{1}\hat{F}^{\ast}_{y, \J}\left(\frac{\hat{\a}}{j_{1}}\right)\right)}}}\nonumber,
\end{align}
and since
\[e\left(-dj_{1}\hat{F}_{y, \J}\left(\frac{\hat{\a}}{j_{1}}\right)\right)=\overline{e\left(dj_{1}\hat{F}_{y, \J}\left(\frac{\hat{\a}}{j_{1}}\right)\right)}, \quad e(-da_{n}y)=\overline{e(da_{n}y)}\]
we have
\begin{align}
&\left|\sum_{|a_{n}|\leq 2j_{1}L}{\sum_{1\leq |d|\leq D}{\frac{D-|d|}{D^{2}} e(da_{n}y)\sum_{\hat{\a}\in \Z^{n-1}}{\omega_{\J}^{\ast}\left(\frac{\hat{\a}}{j_{1}}\right)e\left(dj_{1}\hat{F}^{\ast}_{y, \J}\left(\frac{\hat{\a}}{j_{1}}\right)\right)}}}\right|\\
&\qquad\leq 2 \left|\sum_{|a_{n}|\leq 2j_{1}L}{\sum_{d=1}^{D}{\frac{D-d}{D^{2}} e(da_{n}y)\sum_{\hat{\a}\in \Z^{n-1}}{\omega_{\J}^{\ast}\left(\frac{\hat{\a}}{j_{1}}\right)e\left(dj_{1}\hat{F}^{\ast}_{y, \J}\left(\frac{\hat{\a}}{j_{1}}\right)\right)}}}\right|.\nonumber
\end{align}
Applying the $n$-dimensional Poisson summation formula to the inner most sum yields
\begin{align}
&\sum_{\hat{\a}\in \Z^{n-1}}{\omega_{\J}^{\ast}\left(\frac{\hat{\a}}{j_{1}}\right)e\left(dj_{1}\hat{F}^{\ast}_{y, \J}\left(\frac{\hat{\a}}{j_{1}}\right)\right)}\\
=&\sum_{\hat{\K}\in \Z^{n-1}}{\int_{\R^{n-1}}{\omega_{\J}^{\ast}\left(\frac{\hat{\z}}{j_{1}}\right)e\left(dj_{1}\hat{F}^{\ast}_{y, \J}\left(\frac{\hat{\z}}{j_{1}}\right)-\hat{\K} \cdot\hat{\z}\right)}\D \hat{\z}}\nonumber\\
=&j_{1}^{n-1}\sum_{\hat{\K}\in \Z^{n-1}}{I_{0}(d; \J; \hat{\K})}\nonumber
\end{align}
where
\[I_{0}(d; \J; \hat{\K})=\int_{\R^{n-1}}{\omega_{\J}^{\ast}(\hat{\X})e\left(dj_{1}\hat{F}^{\ast}_{y, \J}(\hat{\X})-j_{1}\hat{\K} \cdot\hat{\X}\right)\D \hat{\X}}.\]
Therefore to obtain a bound for $\mathscr{M}(J, T^{-1})$ it is sufficient to provide a bound for
\begin{equation}
\left|\sum_{\substack{1\leq j_{1}\leq J\\ 0\leq j_{r}\leq \min\{J_{r}, j_{1}\}\\ 2\leq r\leq R}}{\sum_{|a_{n}|\leq 2j_{1}L}{\sum_{d=1}^{D}{\frac{D-d}{D^{2}} e(da_{n}y)j_{1}^{n-1}\sum_{\hat{\a}\in \Z^{n-1}}{I_{0}(d; \J; \hat{\K})}}}}\right|.
\end{equation}
Since $\nabla \hat{F}_{y, \J}^{\ast}=(\nabla \hat{F}_{y, \J})^{-1}$ and $\nabla \hat{F}_{y, \J}$ is a diffeomorphism on $\overline{\mathscr{D}}$ we have that $\nabla \hat{F}^{\ast}_{y, \J}$ is a diffeomorphism on $\nabla \hat{F}_{y, \J}(\mathscr{D})$ and $\nabla \hat{F}_{y, \J}^{\ast}(V_{y, \J}^{+})=U_{y}^{+}$ . Let
\begin{equation}\label{rho'}
\rho'=\frac{\dist(\del \mathscr{D}, \del U_{y})}{2}.
\end{equation}
We repeat the technique of Section \ref{secbounds} and split the set $\Z^{n-1}$ into three disjoint subsets. Let
\[\mathscr{K}_{1}=\left\{\hat{\K}\in \Z^{n-1}\left| \frac{\hat{\K}}{d} \in U_{y}\right.\right\}, \]
\[\mathscr{K}_{2}=\left\{\hat{\K}\in \Z^{n-1}\left| \dist\left(\frac{\hat{\K}}{d}, U_{y}\right)\geq \rho'\right.\right\}\]
and
\[\mathscr{K}_{3}=\Z^{n-1}\backslash (\mathscr{K}_{1} \cup \mathscr{K}_{2}).\]
For each $1\leq i\leq 3$ we define
\begin{align}\label{Mi}
M_{i}&=\sum_{d=1}^{D}{\frac{D-d}{D^{2}}\left| \sum_{\substack{1\leq j_{1}\leq J\\ 0\leq j_{r}\leq \min\{J_{r}, j_{1}\\ 2\leq r\leq R}}{\sum_{|a_{n}|\leq 2j_{1}L}{ e(da_{n}y) j_{1}^{n-1} \sum_{\hat{\K}\in \mathscr{K}_{i}}{I_{0}(d; \J; \hat{\K})} } } \right|}\\
&\ll_{L} \sum_{\substack{1\leq j_{1}\leq J\\ 0\leq j_{r}\leq \min\{J_{r}, j_{1}\}\\ 2\leq r\leq R}}{\sum_{d=1}^{D}{\frac{D-d}{D^{2}}\left| j_{1}^{n}\sum_{\hat{\K}\in \mathscr{K}_{i}}{I_{0}(d; \J; \hat{\K})}\right|}}\nonumber,
\end{align}
such that
\begin{equation}\label{Mestiamte}
\mathscr{M}(J, T^{-1})\ll \left(\prod_{r=2}^{R}J_{r}\right) \frac{J^{n+1}}{D}+M_{1}+M_{2}+M_{3},
\end{equation}
and seek to bound each $M_{i}$ seperately. 
\begin{xrem}
Note that we deliberately choose to estimate the term $e(da_{n}y)$ in (\ref{Mi}) trivially. Considering its contribution could possibly lead to further results.
\end{xrem}
\paragraph{Case $\hat{\K}\in \mathscr{K}_{2}$.} Define
\[\varphi_{1}(\hat{\X})=\frac{d\hat{F}_{y, \J}^{\ast}(\hat{\X})-\hat{\K}\cdot \hat{\X}}{\dist(\hat{\K}, dU_{y})}\]
and
\[\lambda_{1}=j_{1}\dist(\hat{\K}, dU_{y}).\]
Then for all $\hat{\X}\in V_{y, \J}$
\[|\nabla \varphi_{1}(\hat{\X})|=\frac{|d\hat{F}_{y, \J}^{\ast}(\X)-\hat{\K}|}{\dist(\hat{\K}, U)}\geq 1\]
and like in  (\ref{rhodistance}) we conclude
\[|\nabla \varphi_{1}(\hat{\X})|\geq \frac{1}{2}\]
for $\hat{\X}\in V_{y, \J}^{+}$. Next we establish upper bounds for the derivatives of $\varphi_{1}$ and $\omega_{\J}^{\ast}$. In order to do so, we estbalish bounds for the derivatives of $\hat{F}_{y, \J}^{\ast}$ first.

\begin{lem}\label{F*bounds}
Let $i_{1}, ..., i_{n-1}\in \Z_{\geq 0}$ with $\sum_{\mu=1}^{n-1}{i_{\mu}}\leq \ell$. Then for all $\hat{\X}\in U_{y}^{+}$ we have
\[\left|\frac{\del^{i_{1}+\cdots + i_{n-1}}\hat{F}_{y, \J}^{\ast}}{\del^{i_{1}}_{x_{1}}\cdots \del^{i_{n-1}}_{x_{n-1}}}(\hat{\X})\right|\ll 1,\]
where the implicit constant depends only on $(i_{1}, ..., i_{n-1})$, $\rho'$, $\tau$ and upper bounds for (the absolute values of) finitely many derivatives of $f_{r}$ on $U_{y}^{+}\times \mathscr{Y}$ for $1\leq r\leq R$.
\end{lem}

\begin{proof}
For $\hat{\X}\in V_{y, \J}^{+}$ and $\hat{\z}\in U_{y}^{+}$ we have
\[|\hat{\X}\cdot \hat{\z}|+|\hat{F}_{y, \J}|\ll 1,\]
hence we easily deduce $|\hat{F}_{y, \J}^{\ast}|\ll 1$ with (\ref{legendrepoint}).
Recall that $\nabla \hat{F}_{y, \J}^{\ast}=(\nabla \hat{F}_{y, \J})^{-1}$ and $U_{y}^{+}=(\nabla \hat{F}_{y, \J})^{-1}(V_{y, \J}^{+})$. Hence $|\nabla \hat{F}^{\ast}_{y, \J}(\hat{\X})|\ll 1$ for $\hat{\X}\in V_{y, \J}^{+}$. For $\hat{\X}=\nabla \hat{F}_{y, \J}(\hat{\z})$ with $\hat{\z}\in U_{y}^{+}$ we have
\begin{equation}\label{Jacs}
\Jac_{\nabla \hat{F}^{\ast}_{y, \J}}(\hat{\X})=\Jac_{(\nabla \hat{F}_{y, \J})^{-1}}(\hat{\X})=(\Jac_{\nabla \hat{F}_{y, \J}}(\hat{\z}))^{-1},
\end{equation}
where $\Jac_{f}$ denotes the Jacobian matrix of the function $f$ and we used the chain rule. Consequently every second partial derivative of $\hat{F}_{y, \J}^{\ast}$, i.e. the entries of the Jacobian matrix, can be written as
\begin{equation}\label{secondderiv}
\frac{P}{\det(\Jac_{\nabla \hat{F}_{y, \J}}(\hat{\z}))},
\end{equation}
where $P$ is a polynomial expression in the terms of the entries of $\Jac_{\nabla \hat{F}_{y, \J}}(\hat{\z})$. Note that $P$ has degree $(n-1)$ and each coefficient can only be $\pm 1$ or $0$. Since 
\[\left|\frac{\del^{i_{1}+\cdots + i_{n-1}}\hat{F}_{y, \J}}{\del^{i_{1}}_{x_{1}}\cdots \del^{i_{n-1}}_{x_{n-1}}}(\hat{\z})\right|\ll 1\]
is obvious for any $i_{1}, ..., i_{n-1}\in \Z_{\geq 0}$ with $\sum_{\mu=1}^{n-1}i_{\mu}\leq \ell$ and $\hat{\z}\in U_{y}^{+}$, and $\Jac_{\nabla \hat{F}_{y, \J}}=H_{\hat{F}_{y, \J}}$, the desired bounds for the second derivatives follows directly with (\ref{secondderiv}) and (\ref{determinantbound}). Essentially the same idea will be used to argue for higher partial derivatives. Note that for any $k\in \N$ we can express the $k$-th partial derivative with respect to the $\hat{\X}$-variables of an entry in $\Jac_{\nabla \hat{F}_{y, \J}^{\ast}}(\hat{\X})$ as a real polynomial with coefficients independant of $\J$ in terms of:
\begin{compactenum}[(i)]
\item $(\Jac_{\nabla \hat{F}_{y, \J}}(\hat{\z}))^{-m}$, where $m\leq k+1$;
\item entries of $\Jac_{\nabla \hat{F}_{y, \J}}(\hat{\z})$;
\item $m$-th partial derivatives with respect to the $\hat{\z}$-variables of entries in $\Jac_{\nabla \hat{F}_{y, \J}}$, where $m\geq k$;
\item $m$-th partial derivatives with respect to the $\hat{\X}$-variables of entries in $\nabla \hat{F}_{y, \J}^{\ast}=(\nabla \hat{F}_{y, \J})^{-1}(\hat{\X})$, where $m\geq k$.
\end{compactenum}
Now, using (\ref{determinantbound}) again, the desired result follows inductively.
\end{proof}

Now with similar arguments as in \ref{phasederivbound} we can deduce the following.

\begin{cor}\label{phasederivbound*}
Let $i_{1}, ..., i_{n-1}\in \Z_{\geq 0}$ with $\sum_{\mu=1}^{n-1}{i_{\mu}}\leq \ell$. Then for all $\hat{\X}\in U_{y}^{+}$ we have
\[\left|\frac{\del^{i_{1}+\cdots + i_{n-1}}\varphi_{1}}{\del^{i_{1}}_{x_{1}}\cdots \del^{i_{n-1}}_{x_{n-1}}}(\hat{\X})\right|\ll 1,\]
where the implicit constant depends only on $(i_{1}, ..., i_{n-1})$, $\rho'$, $\tau$ and upper bounds for (the absolute values of) finitely many derivatives of $f_{r}$ on $U_{y}^{+}\times \mathscr{Y}$ for $1\leq r\leq R$.
\end{cor}

Recall that $\omega^{\ast}_{\J}=\omega \circ \nabla\hat{F}_{y, \J}^{\ast}$, hence we also obtain the following.

\begin{cor}\label{omega*bound}
Let $i_{1}, ..., i_{n-1}\in \Z_{\geq 0}$ with $\sum_{\mu=1}^{n-1}{i_{\mu}}\leq \ell-1$. Then for all $\hat{\X}\in U_{y}^{+}$ we have
\[\left|\frac{\del^{i_{1}+\cdots + i_{n-1}}\omega_{\J}^{\ast}}{\del^{i_{1}}_{x_{1}}\cdots \del^{i_{n-1}}_{x_{n-1}}}(\hat{\X})\right|\ll 1,\]
where the implicit constant depends only on $(i_{1}, ..., i_{n-1})$, $\rho'$, $\tau$ and upper bounds for (the absolute values of) finitely many derivatives of $f_{r}$ on $U_{y}^{+}\times \mathscr{Y}$ for $1\leq r\leq R$.
\end{cor}

Apllying \ref{stat} with $\varphi=\varphi_{1}$ and $\lambda=\lambda_{1}$ as defined above now yields
\[I_{0}(d; \J; \hat{\K})\ll \lambda^{-\ell +1}=(j_{1}\dist(\hat{\K}, dU_{y}))^{-\ell +1},\]
where the implicit constant is independant of $y$, $\J$ and $\hat{\K}$. Now since $\ell-n\geq 1$ we find similarly to (\ref{K2argument})
\begin{align}
\sum_{\hat{\K}\in \mathscr{K}_{2}}{I_{0}(d; \J; \hat{\K})}&\ll j_{1}^{-\ell+1}\sum_{\hat{\K}\in \mathscr{K}_{2}}{\dist(\hat{\K}, dU_{y})^{{-\ell+1}}}\\
&\ll j_{1}^{-\ell+1}\sum_{m=0}^{\infty}{\sum_{\substack{\hat{\K}\in \Z^{n-1}\\ 2^{m}j_{1}\rho'\leq \dist(\hat{\K}, dU_{y})<2^{m+1}j_{1}\rho'}}}{\frac{1}{(2^{m}j_{1}\rho')^{\ell -1}}},\nonumber\\
&\ll j_{1}^{-\ell+1}\sum_{m=0}^{\infty}{\frac{(Lj_{1}+2^{m+1}j_{1}\rho')^{n-1}}{2^{m}j_{1}\rho'}},\nonumber\\
&\ll j_{1}^{-\ell +1}\nonumber,
\end{align}
where the implicit constant is independant of $y$ and $d$. Consequently we obtain
\begin{align}\label{M2}
M_{2}&\leq \sum_{\substack{1\leq j_{1}\leq J\\ 0\leq j_{r}\leq \min\{J_{r}, j_{1}\}\\ 2\leq r\leq R}}{\sum_{d=1}^{D}{\frac{D-d}{D^{2}}j_{1}^{n-1}\left|\sum_{\hat{\K}\in \mathscr{K}_{i}}{I_{0}(d; \J; \hat{\K})}\right|}}\\
&\ll \frac{D-1}{2D}\sum_{\substack{1\leq j_{1}\leq J\\ 0\leq j_{r}\leq \min\{J_{r}, j_{1}\}\\ 2\leq r\leq R}}{j_{1}^{n-\ell+1}}\nonumber\\
&\ll \left(\prod_{r=2}^{R}{J_{r}}\right)\log J.\nonumber
\end{align}

\paragraph{Case $\hat{\K}\in \mathscr{K}_{3}$.} Let $\lambda=j_{1}d$ and
\[\varphi(\hat{\X})=\hat{F}_{y, \J}^{\ast}(\hat{\X})-\frac{\hat{\K}}{d}\cdot \hat{\X}.\]
By definition, for each fixed $d$ we have that $\hat{\K}\in d\mathscr{D}$ determines a unique preimage
\[\xcritd=(\nabla \hat{F}_{y, \J}^{\ast})^{-1}\left(\frac{\hat{\K}}{d}\right)=\nabla\hat{F}_{y, \J}\left(\frac{\hat{\K}}{d}\right)\]
that is also a critical point for $\varphi$ in the sense that
\[\nabla\varphi(\xcritd)=\nabla \hat{F}_{y, \J}^{\ast}(\xcritd)-\frac{\hat{\K}}{d}=\0.\]

\begin{lem}\label{nonstatgradientbound2}
Let $\hat{\X}\in \nabla\hat{F}_{y, \J}(\mathscr{D}_{+})\backslash\{\xcritd\}$. Then 
\[\frac{|\hat{\X}-\xcritd|}{|\nabla \varphi(\hat{\X})|}\ll 1\]
where the implicit constant is independant of $d$, $\J$ and $\hat{\K}$.
\end{lem}

\begin{proof}
Observe that for $\hat{\K}/d\in \mathscr{D}$ we have
\[\frac{|\hat{\X}-\xcritd|}{|\nabla\varphi(\hat{\X})|}=\frac{|\hat{\X}-(\nabla\hat{F}_{y, \J}^{\ast})^{-1}(\hat{\K}/d)|}{|\nabla \hat{F}_{y, \J}^{\ast}(\hat{\X})-\frac{\hat{\K}}{d}|},\]
hence it is sufficient to prove
\[\frac{|\hat{\X}-\hat{\z}|}{|\nabla \hat{F}_{y, \J}^{\ast}(\hat{\X})-\nabla \hat{F}_{y, \J}^{\ast}(\hat{\z})|}\ll 1\]
for $\hat{X}, \hat{\z}\in \overline{\nabla \hat{F}_{y, \J}(\mathscr{D}_{+})}$ and $\hat{\X}\neq \hat{\z}$. Taking $\hat{\X}', \hat{\z}'\in \overline{\mathscr{D}_{+}}$ with $\hat{\X}'\neq\hat{\z}'$ such that $\hat{\X}=\nabla\hat{F}_{y, \J}(\hat{\X}')=(\nabla\hat{F}_{y, \J}^{\ast})^{-1}$ (and the same for $\hat{\z}$) the inequality is euqivalent to 
\[\frac{|\nabla \hat{F}_{y, \J}(\hat{\X}')-\nabla \hat{F}_{y, \J}(\hat{\z}')|}{|\hat{\X}'-\hat{\z}'|}\ll 1.\]
or alternatively
\[1\ll \frac{|\hat{\X}'-\hat{\z}'|}{|\nabla \hat{F}_{y, \J}(\hat{\X}')-\nabla \hat{F}_{y, \J}(\hat{\z}')|}.\]
We have already established in the proof of \ref{nonstatgradientbound} that
\[\nabla \hat{F}_{y, \J}(\hat{\X}')-\nabla \hat{F}_{y, \J}(\hat{\z}')=H_{\hat{F}_{y, \J}}(\hat{\z}')(\hat{\X}'-\hat{\z}')+O(|\hat{\X}-\hat{\z'}|^{2}),\]
which yields the desired lower bound immediately.
\end{proof}
Note that because of \ref{F*bounds} we can deduce the same result from \ref{phasederivbound*} for $\varphi$ in this case, i.e. for given $i_{1}, ..., i_{n-1}\in \Z_{\geq 0}$ with $\sum_{\mu=1}^{n-1}{i_{\mu}}\leq \ell$ and $\hat{\X}\in V_{y, \J}^{+}$ we have
\begin{equation}\label{phasederivbound**}
\left|\frac{\del^{i_{1}+\cdots + i_{n-1}}\varphi}{\del^{i_{1}}_{x_{1}}\cdots \del^{i_{n-1}}_{x_{n-1}}}(\hat{\X})\right|\ll 1.
\end{equation}
The implicit constant is again independent of $y$, $d$, $\J$ and such $\hat{\K}$ that satisfy this case. By construction we have $H_{\varphi}=H_{\hat{F}_{y, \J}^{\ast}}$, so with (\ref{legendrematrix}) and (\ref{determinantbound}) we have $H_{\varphi}(\xcritd)\neq 0$ and consequently Lemma \ref{stat} yields
\[I_{0}(d;\J;\hat{\K})\ll \lambda^{-\frac{n+1}{2}}=j_{1}^{-\frac{n}{2}-\frac{1}{2}}d^{-\frac{n}{2}-\frac{1}{2}}.\]
Note that we chose $\hat{\K}/d\not\in U_{y}$, hence $\xcritd\not\in \nabla\hat{F}_{y, \J}(U_{y})$, i.e. $\omega_{\J}^{\ast}(\xcritd)=0$. With a similar argument as in \ref{countK} we obtain
\[\sum_{\hat{\K}\in \mathscr{K}_{3}}{I_{0}(d;\J; \hat{\K})}\ll d^{n-1}j_{1}^{-\frac{n}{2}-\frac{1}{2}}d^{-\frac{n}{2}-\frac{1}{2}}=j_{1}^{-\frac{n}{2}-\frac{1}{2}}d^{\frac{n}{2}-\frac{3}{2}}.\]
Hence we have
\begin{align}\label{M3final}
M_{3}&\leq \sum_{\substack{1\leq j_{1}\leq J\\ 0\leq j_{r}\leq \min\{J_{r}, j_{1}\}\\ 2\leq r\leq R}}{\sum_{d=1}^{D}{\frac{D-d}{D^{2}}\left| j_{1}^{n}\sum_{\hat{\K}\in \mathscr{K}_{i}}{I_{0}(d; \J; \hat{\K})}\right|}}&\\
&\ll \frac{1}{D}\sum_{\substack{1\leq j_{1}\leq J\\ 0\leq j_{r}\leq \min\{J_{r}, j_{1}\}\\ 2\leq r\leq R}}{ j_{1}^{\frac{n}{2}-\frac{1}{2}}\sum_{d=1}^{D}{d^{\frac{n}{2}-\frac{3}{2}} }}\nonumber\\
&\ll \left(\prod_{r=2}^{R}{J_{r}}\right)J^{\frac{n}{2}+\frac{1}{2}}D^{\frac{n}{2}-\frac{3}{2}}\nonumber. 
\end{align}

\paragraph{Case $\hat{\K}\in \mathscr{K}_{1}$.} Let $\lambda$ and $\varphi$ be as in the previous case, specifically maintaining Lemma \ref{nonstatgradientbound2} and (\ref{phasederivbound**}). Then we have
\begin{equation}\label{phicrit}
\varphi(\xcritd)=\hat{F}^{\ast}_{y, \J}(\xcritd)-\frac{\hat{\K}}{d}\cdot \xcritd=-\hat{F}_{y, \J}\left(\frac{\hat{\K}}{d}\right)
\end{equation}
and by (\ref{legendrematrix})
\[H_{\hat{F}_{y, \J}^{\ast}}(\xcritd)=H_{\hat{F}_{y, \J}}\left(\frac{\hat{\K}}{d}\right)^{-1}.\]
Similar to the arguments right before (\ref{Integraly}) we find that the signature $\sigma$ of $H_{\varphi}(\xcrit)=H_{\hat{F}_{y, \J}^{\ast}}(\xcritd)$ is constant for all $d$, $\J$ and $\hat{\K}$ in consideration, hence with Lemma \ref{stat}, (\ref{determinantbound}) and (\ref{phicrit}) we obtain
\begin{align}
&I_{0}(d; \J; \hat{\K})\\
&\qquad =\frac{\omega^{\ast}_{\J}(\xcritd)}{|\det H_{\hat{F}_{y, \J}^{\ast}}(\xcritd)|^{\frac{1}{2}}}(j_{1}d)^{-\frac{n}{2}+\frac{1}{2}}e\left(-j_{1}d\hat{F}_{y, \J}\left(\frac{\hat{\K}}{d}\right)+\frac{\sigma}{8}\right)+\bigO{(j_{1}d)^{-\frac{n}{2}-\frac{1}{2}}}\nonumber\\
&\qquad =\omega\left(\frac{\hat{\K}}{d}\right)|\det H_{\hat{F}_{y, \J}}(\hat{\K}/d)|^{\frac{1}{2}}(j_{1}d)^{-\frac{n}{2}+\frac{1}{2}}e\left(-d(j_{1}\hat{f}_{1, y}+\cdots+j_{R}\hat{f}_{R, y})\left(\frac{\hat{\K}}{d}\right)+\frac{\sigma}{8}\right)\nonumber\\
&\qquad+\bigO{(j_{1}d)^{-\frac{n}{2}-\frac{1}{2}}},\nonumber
\end{align}
where the implicit constant is independant of $y$, $d$, $\J$ and $\hat{\K}$.
For $(u_{1}, ..., u_{R})\in \R_{>0}\times \R^{R-1}_{\geq 0}$ we consider the function
\[\Psi_{\hat{\K}; d}(u_{1}, ..., u_{R})=u_{1}^{\frac{n}{2}+\frac{1}{2}}|\det{H_{\hat{f}_{1, y}+\frac{u_{2}}{u_{1}}\hat{f}_{2, y}+\cdots +\frac{u_{R}}{u_{1}}f_{R, y}}}(\hat{\K}/d)|^{\frac{1}{2}}\]
and obtain
\begin{align}
&\left|\sum_{\substack{1\leq j_{1}\leq J\\ 0\leq j_{r}\leq \min\{J_{r}, j_{1}\}\\ 2\leq r\leq R}}{j_{1}^{n}I_{0}(d; \J; \hat{\K})}\right|\\
&\qquad \ll \omega\left(\frac{\hat{\K}}{d}\right)d^{-\frac{n}{2}+\frac{1}{2}}\sum_{\substack{0\leq j_{r}\leq J_{r}\\2\leq r\leq R}}{\left|\sum_{\max\{1, j_{2}, ...,j_{R}\}\leq j_{1}\leq J }{\Psi_{\hat{\K}; d}}(j_{1}, ..., j_{R})e(-dj_{1}\hat{f}_{1, y}(\hat{\K}/d)) \right|}\nonumber\\
&\qquad + \bigO{\left(\prod_{r=2}^{R}{J_{r}}\right)J^{\frac{n}{2}+\frac{1}{2}}d^{-\frac{n}{2}-\frac{1}{2}}}.\nonumber
\end{align}
Given any fixed $u_{2}, ..., u_{R}\in \R_{\geq 0}$ we find that $\Psi(\cdot, u_{2}, ..., u_{R})$ is a smooth function on the set $\{u_{1}\in \R_{>0}\mid u_{1}\geq u_{r},  2\leq r\leq R\}$. Let $\Psi^{(1)}_{\hat{\K}; d}$ denote the partial derivative of $\Psi_{\hat{\K}; d}$ in $u_{1}$-direction. Then by partial summation we have for the innermost sum 
\begin{align}
&\sum_{\max\{1, j_{2}, ...,  j_{R}\}\leq j_{1}\leq J}{\Psi_{\hat{\K}; d}(j_{1}, ...., j_{R})e(-dj_{1}\hat{f}_{1, y}(\hat{\K}/d))}\\
&\leq\Psi_{\hat{\K}; d}(J, j_{2}, ..., j_{R})\sum_{j_{1}=1}^{J}{e(-j_{1}d\hat{f}_{1, y}(\hat{\K}/d))}-\int_{1}^{J}{ \sum_{j_{1}=1}^{\xi}{e(-j_{1}d\hat{f}_{1, y}(\hat{\K}/d))}\Psi^{(1)}_{\hat{\K}; d}(\xi, j_{2}, ..., j_{R})\D \xi}.\nonumber
\end{align}
We distinguish two cases. First if $||d\hat{f}_{1, y}(\hat{\K}/d)||\geq J^{-1}$ we obtain
\begin{align}
&\sum_{\max\{1, j_{2}, ...,  j_{R}\}\leq j_{1}\leq J}{\Psi_{\hat{\K}; d}(j_{1}, ...., j_{R})e(-dj_{1}\hat{f}_{1, y}(\hat{\K}/d))}\\
&\qquad\ll \frac{\Psi_{\hat{\K}; d}(J, j_{2}, ..., j_{R})}{||d\hat{f}_{1, y}(\hat{\K}/d)||}+\frac{1}{||d\hat{f}_{1, y}(\hat{\K}/d)||}\int_{1}^{J}{\Psi_{\hat{\K}; d}^{(1)}(\xi, j_{2}, ..., j_{R})\D \xi}\nonumber\\
&\qquad\ll \frac{\Psi_{\hat{\K}; d}(J, j_{2}, ..., j_{R})}{||d\hat{f}_{1, y}(\hat{\K}/d)||}.\nonumber
\end{align}
On the other hand if $||d\hat{f}_{1, y}(\hat{\K}/d)||<J^{-1}$ we have
\begin{align}
&\sum_{\max\{1, j_{2}, ...,  j_{R}\}\leq j_{1}\leq J}{\Psi_{\hat{\K}; d}(j_{1}, ...., j_{R})e(-dj_{1}\hat{f}_{1, y}(\hat{\K}/d))}\\
&\qquad\ll \Psi_{\hat{\K}; d}(J, j_{2}, ..., j_{R})J+\int_{1}^{J}{\xi\Psi^{(1)}_{\hat{\K}; d}(\xi, j_{2}, ..., j_{R})\D\xi}.\nonumber
\end{align}
To simplify further, we need estimates for $\Psi_{\hat{\K}; d}$ and $\Psi^{(1)}_{\hat{\K}; d}$ respectively.

\begin{lem}
Let $(u_{1}, ..., u_{R})\in \R_{>0}\times \R^{R-1}_{\geq 0}$ be such that $u_{r}\leq u_{1}$ for $2\leq r\leq R$. Then for any $\hat{\K}\in \mathscr{K}_{1}$ we have
\[|\Psi_{\hat{\K}; d}(u_{1}, ..., u_{R})|\ll u_{1}^{\frac{n}{2}+\frac{1}{2}}\]
and
\[|\Psi^{(1)}_{\hat{\K}; d}(u_{1}, ..., u_{R})|\ll u_{1}^{\frac{n}{2}-\frac{1}{2}},\]
where the respective implicit constants are independant of $y$, $d$ and $\hat{\K}$.
\end{lem}

\begin{proof}
The first estimate is an obvious consequence of (\ref{determinantbound}). Write
\[\det H_{\hat{f}_{1, y}+\alpha_{2}\hat{f}_{2, y}+\cdots + \alpha_{R}\hat{f}_{R, y}}=\sum_{\substack{0\leq \nu_{2}+\cdots +\nu_{R}\leq n-1\\0\leq \nu_{2}, ..., \nu_{R}}}{A_{\nu_{2}, ..., \nu_{R}}\alpha_{2}^{\nu_{2}}\cdots \alpha_{R}^{\nu_{R}}}.\]
Then since $\hat{\K}/d\in U_{y}$ we have $|A_{\nu_{2}, ..., \nu_{R}}|\ll 1$, where the implicit constant is independant of $y$, $d$ and $\hat{\K}$. Now by product and chain rule
\begin{align*}
&|\Psi_{\hat{\K}; d}(u_{1}, ..., u_{R})|\\
&\ll \frac{n+1}{2}u_{1}^{\frac{n}{2}-\frac{1}{2}}|\det{H_{\hat{f}_{1, y}+\frac{u_{2}}{u_{1}}\hat{f}_{2, y}+\cdots +\frac{u_{R}}{u_{1}}f_{R, y}}}(\hat{\K}/d)|^{\frac{1}{2}}\\
&+\frac{u_{1}^{\frac{n}{2}-\frac{1}{2}}}{2|\det{H_{\hat{f}_{1, y}+\frac{u_{2}}{u_{1}}\hat{f}_{2, y}+\cdots +\frac{u_{R}}{u_{1}}f_{R, y}}}(\hat{\K}/d)|^{\frac{1}{2}}}\sum_{\substack{0\leq \nu_{2}+\cdots +\nu_{R}\leq n-1\\0\leq \nu_{2}, ..., \nu_{R}}}{|A_{\nu_{2}, ..., \nu_{R}}|\left(\frac{u_{2}}{u_{1}}\right)^{\nu_{2}}\cdots \left(\frac{u_{R}}{u_{1}}\right)^{\nu_{R}}}\\
& \ll u_{1}^{\frac{n}{2}-\frac{1}{2}},
\end{align*}
since $0\leq u_{2}, ..., u_{R}\leq u_{1}$ and $u_{1}>0$ by assumption. 
\end{proof}
Therefore we obtain
\begin{align}\label{M1split}
M_{1}&\ll \frac{1}{D}\sum_{d=1}^{D}{\sum_{\substack{\hat{\K}\in\Z^{n-1}\\||d\hat{f}_{1, y}(\hat{\K}/d)||\leq J^{-1}}}{\omega\left(\frac{\hat{\K}}{d}\right)d^{-\frac{n}{2}+\frac{1}{2}}\left(\prod_{r=2}^{R}{J_{r}}\right) J^{\frac{n}{2}+\frac{3}{2}}} }\\
&+\frac{1}{D}\sum_{d=1}^{D}{\sum_{\substack{\hat{\K}\in \Z^{n-1}\\ J^{-1}<||d\hat{f}_{1, y}(\hat{\K}/d)||}}{\omega\left(\frac{\hat{\K}}{d}\right)d^{-\frac{n}{2}+\frac{1}{2}}\left(\prod_{r=2}^{R}{J_{r}}\right) J^{\frac{n}{2}+\frac{1}{2}}}||d\hat{f}_{1, y}(\hat{\K}/d)||^{-1} }\nonumber\\
&+\frac{1}{D}\sum_{d=1}^{D}{\sum_{\hat{\K}\in \mathscr{K}_{1}}{\left(\prod_{r=2}^{R}{J_{r}}\right)J^{\frac{n}{2}+\frac{1}{2}} d^{-\frac{n}{2}-\frac{1}{2}}}}.\nonumber
\end{align}
With a similar argument as in \ref{countK} we can bound the last term in (\ref{M1split}) by
\begin{align}\label{M1part3}
\frac{1}{D}\sum_{d=1}^{D}{\sum_{\hat{\K}\in \mathscr{K}_{1}}{\left(\prod_{r=2}^{R}{J_{r}}\right)J^{\frac{n}{2}+\frac{1}{2}} d^{-\frac{n}{2}-\frac{1}{2}}}}&\ll \left(\prod_{r=2}^{R}{J_{r}}\right)\frac{J^{\frac{n}{2}+\frac{1}{2}}}{D}\sum_{d=1}^{D}{d^{\frac{n}{2}-\frac{3}{2}}}\\
&\ll \left(\prod_{r=2}^{R}{J_{r}}\right)J^{\frac{n}{2}+\frac{1}{2}}D^{\frac{n}{2}-\frac{3}{2}}.\nonumber
\end{align}
In order to estiamte the second term in (\ref{M1split}) we want to sum dyadically. Note that since $||d\hat{f}_{1, y}(\hat{\K}/d)||\leq\frac{1}{2}$ we can assume $J^{-1}\leq \frac{1}{2}$ and obtain
\begin{align}\label{M1part2}
&\frac{1}{D}\sum_{d=1}^{D}{\sum_{\substack{\hat{\K}\in \Z^{n-1}\\ J^{-1}<||d\hat{f}_{1, y}(\hat{\K}/d)||}}{\omega\left(\frac{\hat{\K}}{d}\right)d^{-\frac{n}{2}+\frac{1}{2}}\left(\prod_{r=2}^{R}{J_{r}}\right) J^{\frac{n}{2}+\frac{1}{2}}}||d\hat{f}_{1, y}(\hat{\K}/d)||^{-1} }\\
&\qquad\leq \left(\prod_{r=2}^{R}{J_{r}}\right)\frac{J^{\frac{n}{2}+\frac{1}{2}}}{D}\sum_{i=1}^{\frac{\log J}{\log 2}+1}{J2^{1-i}\sum_{d=1}^{D}{d^{-\frac{n}{2}+\frac{1}{2}}\sum_{\substack{\hat{\K}\in \Z^{n-1}\\ \frac{2^{i-1}}{J}<||d\hat{f}_{1, y}(\hat{\K}/d)||\leq \frac{2^{i}}{J}}}{\omega\left(\frac{\hat{\K}}{d}\right)}}}.\nonumber
\end{align}
Now for both the first term in (\ref{M1split}) and (\ref{M1part2}) we utilize the following result from \cite[Theorem 2]{Hua1}: For any $X>0$ we have
\begin{equation}
\sum_{d=1}^{D}{\sum_{\substack{\hat{\K}\in \Z^{n-1}\\ ||df_{1}(\hat{\K}/d)||\leq X^{-1}}}{\omega\left(\frac{\hat{\K}}{d}\right)}}\ll X^{-1}D^{n}+D^{n-1}\mathscr{E}_{n-1}(D),
\end{equation}
where
\[\mathscr{E}_{m}(D)=\mathscr{E}_{m}^{(\idl{c}_{3}; \idl{c}_{4})}(D)=\left\{\begin{array}{*{2}{c}} \exp(\idl{c}_{3}\sqrt{\log D})& \mbox{if $m=2$}\\ (\log D)^{\idl{c}_{4}}& \mbox{if $m\geq 3$}\end{array}\right..\]
for some positive constans $\idl{c}_{3}$ and $\idl{c}_{4}$. Here the implicit constants aswell as $\idl{c}_{3}$ and $\idl{c}_{4}$ only depend on $n, c_{1}$ and $c_{2}$ in (\ref{determinantbound}), $\rho$ in (\ref{rho}), $\rho'$ in (\ref{rho'}) and upper bounds for (the absolute values) of finitely many derivatives of $\omega$ and $f_{1}$ on $\mathscr{D}_{+}\times \mathscr{Y}$. In particular, they are independent of $y$.\\
By partial summation we find that
\begin{equation}\label{hypersurfacebound}
\sum_{d=1}^{D}{\sum_{\substack{\hat{\K}\in \Z^{n-1}\\||d\hat{f}_{1, y}(\hat{\K}/d)||\leq X^{-1}}}{\omega\left(\frac{\hat{\K}}{d}\right) d^{-\frac{n}{2}+\frac{1}{2}}}}\ll D^{-\frac{n}{2}+\frac{1}{2}}(X^{-1}D^{n}+D^{n-1}\mathscr{E}_{n-1}(D)). 
\end{equation}
Therefore we can estimate the first term in (\ref{M1split}) by
\begin{align}\label{M1part1}
&\frac{1}{D}\sum_{d=1}^{D}{\sum_{\substack{\hat{\K}\in\Z^{n-1}\\||d\hat{f}_{1, y}(\hat{\K}/d)||\leq J^{-1}}}{\omega\left(\frac{\hat{\K}}{d}\right)d^{-\frac{n}{2}+\frac{1}{2}}\left(\prod_{r=2}^{R}{J_{r}}\right) J^{\frac{n}{2}+\frac{3}{2}}} }\\
&\ll \left(\prod_{r=2}^{R}{J_{r}}\right)\frac{J^{\frac{n}{2}+\frac{3}{2}}}{D}D^{-\frac{n}{2}+\frac{1}{2}}(J^{-1}D^{n}+D^{n-1}\mathscr{E}_{n-1}(D))\nonumber\\
&\ll \left(\prod_{r=2}^{R}{J_{r}}\right)(J^{\frac{n}{2}+\frac{1}{2}}D^{\frac{n}{2}-\frac{1}{2}}+J^{\frac{n}{2}+\frac{3}{2}}D^{\frac{n}{2}-\frac{3}{2}}\mathscr{E}_{n-1}(D))\nonumber
\end{align}
and (\ref{M1part2}) by
\begin{align}\label{M1part2final}
&\left(\prod_{r=2}^{R}{J_{r}}\right)\frac{J^{\frac{n}{2}+\frac{1}{2}}}{D}\sum_{i=1}^{\frac{\log J}{\log 2}+1}{J2^{1-i}\sum_{d=1}^{D}{d^{-\frac{n}{2}+\frac{1}{2}}\sum_{\substack{\hat{\K}\in \Z^{n-1}\\ \frac{2^{i-1}}{J}<||d\hat{f}_{1, y}(\hat{\K}/d)||\leq \frac{2^{i}}{J}}}{\omega\left(\frac{\hat{\K}}{d}\right)}}}\\&\ll \left(\prod_{r=2}^{R}{J_{r}}\right)\frac{J^{\frac{n}{2}+\frac{1}{2}}}{D}\sum_{i=1}^{\frac{\log J}{\log 2}+1}{J2^{1-i}D^{-\frac{n}{2}+\frac{1}{2}}(2^{i}J^{-1}D^{n}+D^{n-1}\mathscr{E}_{n-1}(D))}\nonumber\\
&\ll \left(\prod_{r=2}^{R}{J_{r}}\right)\frac{J^{\frac{n}{2}+\frac{1}{2}}}{D}((\log J)D^{\frac{n}{2}+\frac{1}{2}}+JD^{\frac{n}{2}-\frac{1}{2}}\mathscr{E}_{n-1}(D))\nonumber\\
&\ll \left(\prod_{r=2}^{R}{J_{r}}\right)((\log J)J^{\frac{n}{2}+\frac{1}{2}}D^{\frac{n}{2}-\frac{1}{2}}+J^{\frac{n}{2}+\frac{3}{2}}D^{\frac{n}{2}-\frac{3}{2}}\mathscr{E}_{n-1}(D)).\nonumber
\end{align}
Putting together (\ref{M1part3}), (\ref{M1part1}) and (\ref{M1part2final}) yields
\begin{align}\label{M1final}
M_{1}&\ll \left(\prod_{r=2}^{R}{J_{r}}\right)(J^{\frac{n}{2}+\frac{1}{2}}D^{\frac{n}{2}-\frac{1}{2}}+J^{\frac{n}{2}+\frac{3}{2}}D^{\frac{n}{2}-\frac{3}{2}}\mathscr{E}_{n-1}(D)) \\
&+\left(\prod_{r=2}^{R}{J_{r}}\right)((\log J)J^{\frac{n}{2}+\frac{1}{2}}D^{\frac{n}{2}-\frac{1}{2}}+J^{\frac{n}{2}+\frac{3}{2}}D^{\frac{n}{2}-\frac{3}{2}}\mathscr{E}_{n-1}(D))\nonumber\\
&+ \left(\prod_{r=2}^{R}{J_{r}}\right)J^{\frac{n}{2}+\frac{1}{2}}D^{\frac{n}{2}-\frac{3}{2}}\nonumber\\
&\ll \left(\prod_{r=2}^{R}{J_{r}}\right)((\log J)J^{\frac{n}{2}+\frac{1}{2}}D^{\frac{n}{2}-\frac{1}{2}}+J^{\frac{n}{2}+\frac{3}{2}}D^{\frac{n}{2}-\frac{3}{2}}\mathscr{E}_{n-1}(D)).\nonumber
\end{align}
\paragraph{Final estimate.} Recall $D=\subgauss{T/2}$ and $T\geq 2$. By combining (\ref{Mestiamte}), (\ref{M2}), (\ref{M3final}) and (\ref{M1final}) we obtain
\begin{align}
\mathscr{M}(J, T^{-1})&\ll \left(\prod_{r=2}^{R}{J_{r}}\right)\frac{J^{n+1}}{T}+\left(\prod_{r=2}^{R}{J_{r}}\right)((\log J)J^{\frac{n}{2}+\frac{1}{2}}T^{\frac{n}{2}-\frac{1}{2}}+J^{\frac{n}{2}+\frac{3}{2}}T^{\frac{n}{2}-\frac{3}{2}}\mathscr{E}_{n-1}(T))\\
&+\left(\prod_{r=2}^{R}{J_{r}}\right)\log J +\left(\prod_{r=2}^{R}{J_{r}}\right)J^{\frac{n}{2}+\frac{1}{2}}T^{\frac{n}{2}-\frac{3}{2}}\nonumber\\
&\ll \left(\prod_{r=2}^{R}{J_{r}}\right)(J^{n+1}T^{-1}+(\log J)J^{\frac{n}{2}+\frac{1}{2}}T^{\frac{n}{2}-\frac{1}{2}}+J^{\frac{n}{2}+\frac{3}{2}}T^{\frac{n}{2}-\frac{3}{2}}\mathscr{E}_{n-1}(T)).\nonumber
\end{align}
We distinguish two cases. First if $T^{-1}\leq J^{-1}$ we have
\begin{equation}\label{Mcase1}
\mathscr{M}(J, T^{-1})\leq \mathscr{M}(J, J^{-1})\ll \left(\prod_{r=2}^{R}{J_{r}}\right)J^{n}((\log J) +\mathscr{E}_{n-1}(J)).
\end{equation}
On the other hand, if $T^{-1}>J$, i.e $J> T$, then
\begin{equation}\label{Mcase2}
\mathscr{M}(J, T^{-1})\ll \left(\prod_{r=2}^{R}{J_{r}}\right)(J^{n+1}T^{-1}+J^{n}((\log J)+\mathscr{E}_{n-1}(J))).
\end{equation}
Therefore we conclude
\[\sum_{\substack{1\leq j_{1}\leq J\\0\leq j_{r}\leq \min\{J_{r}, j_{1}\} \\2\leq r\leq R}}{\sum_{\substack{\K\in \Z^{n}\\ ||j_{1}\varphi_{y}(\xcrit)||<T^{-1}}}{\omega^{\ast}_{\J}\left(\frac{\hat{\K}}{j_{1}}\right)}}=\mathscr{M}(J, T^{-1})\ll \left(\prod_{r=2}^{R}{J_{r}}\right)(J^{n+1}T^{-1}+J^{n}\mathscr{E}_{n-1}(J)).\]
Recall that $\omega^{\ast}_{\J}(\hat{\K}/j_{1})=\omega\circ (\nabla \hat{F}_{y, \J})^{-1}(\hat{\K}/j_{1})=\omega(\xcrit)$, hence \ref{essential} follows.

\section{Proof of Theorem \ref{mainthm}}\label{maintheoremsec}
Recall (\ref{Nsplit}), hence with the bounds obtained for $N_{1}, N_{2}$ and $N_{3}$ in Section \ref{secbounds}, i.e. (\ref{case1}), (\ref{case2}) and (\ref{case3}), we have
\begin{align}
N^{(1; (1, ..., 1), n)}(Q, \delta)&\ll N_{1}+N_{2}+N_{3}\\
&\ll (1+\log J)^{R}((\log Q)Q^{\frac{n}{2}+\frac{1}{2}}J^{\frac{n}{2}+\frac{1}{2}}+Q^{\frac{n}{2}+\frac{3}{2}}J^{\frac{n}{2}-\frac{1}{2}}\mathscr{E}_{n-1}(J))\nonumber\\
&+\log Q (1+\log J)^{R}+J^{\frac{n}{2}-\frac{1}{2}}Q^{\frac{n}{2}+\frac{1}{2}}(1+\log J)^{R}.\nonumber\\
&\ll (1+\log J)^{R}((\log Q)Q^{\frac{n}{2}+\frac{1}{2}}J^{\frac{n}{2}+\frac{1}{2}}+Q^{\frac{n}{2}+\frac{3}{2}}J^{\frac{n}{2}-\frac{1}{2}}\mathscr{E}_{n-1}(J))\nonumber
\end{align}
Now with (\ref{maineq}), (\ref{NrQd}) and the remark made right after (\ref{NrQd}) we obtain
\begin{align}
&|N_{\omega}(Q, \delta)-(2\delta)^{R}N_{0}|\\
&\ll \delta^{R-1}\frac{Q^{n+1}}{J}+\frac{Q^{n+1}}{J^{R}}+ (1+\log J)^{R}(\log Q)Q^{\frac{n}{2}+\frac{1}{2}}J^{\frac{n}{2}+\frac{1}{2}}\nonumber\\
&+(1+\log J)^{R}Q^{\frac{n}{2}+\frac{3}{2}}J^{\frac{n}{2}-\frac{1}{2}}\mathscr{E}_{n-1}(J).\nonumber
\end{align}
Note that the constants $\idl{c}_{1}'$ and $\idl{c}_{2}'$ in $\mathscr{E}_{n-1}(J)=\mathscr{E}^{(\idl{c}_{1}'; \idl{c}_{2}')}_{n-1}(J)$ as well as the implicit constants only depend on $n$, $R$, $c_{1}$ and $c_{2}$ in (\ref{determinantbound}), $\rho$ in (\ref{rho}), $\rho'$ in (\ref{rho'}) for each choice of $(r; \bepsilon; \nu)$ ($1\leq r\leq R, \bepsilon\in \{\pm 1\}^{R}, 1\leq \nu\leq n$) and upper bounds for (the absolute values) of finitely many derivatives of $\omega$ and $f_{1}$ on $\mathscr{D}_{+}\times \mathscr{Y}$. Recall that while we only adressed the case $(r; \bepsilon; \nu)$ the bounds are identical in each case. Since we can still choose the parameter $J\geq 1$, consider the equivalences
\begin{align}\label{Jequiv}
Q^{\frac{n}{2}+\frac{1}{2}}J^{\frac{n}{2}+\frac{1}{2}}< Q^{\frac{n}{2}+\frac{3}{2}}J^{\frac{n}{2}-\frac{1}{2}}&\quad \Leftrightarrow \quad J< Q\\
\delta^{R-1}\frac{Q^{n+1}}{J} < \frac{Q^{n+1}}{J^{R}}&\quad\Leftrightarrow\quad  J < \delta^{-1}\nonumber\\
\frac{Q^{n+1}}{J^{R}}\leq Q^{\frac{n}{2}+\frac{3}{2}}J^{\frac{n}{2}-\frac{1}{2}}&\quad \Leftrightarrow\quad Q^{\frac{n-1}{n+2R-1}} \leq J\nonumber\\
\delta^{R-1}\frac{Q^{n+1}}{J}\leq Q^{\frac{n}{2}+\frac{3}{2}}J^{\frac{n}{2}-\frac{1}{2}}&\quad \Leftrightarrow\quad  \delta^{\frac{2(R-1)}{n+1}}Q^{\frac{n-1}{n+1}}\leq J\nonumber
\end{align}
and distinguish two cases. If $\delta^{-1}>Q^{\frac{n-1}{n+2R-1}}$ then let $J=Q^{\frac{n-1}{n+2R-1}}$, so by the first, second and third equivalence in \ref{Jequiv} we have 
\[|N_{\omega}(Q, \delta)-(2\delta)^{R}N_{0}|\ll (\log Q)^{R}Q^{\frac{n^{2}+Rn+3R-1}{n+2R-1}}\mathscr{E}_{n-1}(Q).\]
If $\delta^{-1}\leq Q^{\frac{n-1}{n+2R-1}}$ then let $J=\delta^{\frac{2(R-1)}{n+1}}Q^{\frac{n-1}{n+1}}\geq \delta^{-1}$, so by the second, third and fourth equivalence in \ref{Jequiv} we have
\[|N_{\omega}(Q, \delta)-(2\delta)^{R}N_{0}|\ll \delta^{\frac{(R-1)(n-1)}{n+1}}(\log Q)^{R}Q^{\frac{n^{2}+n+2}{n+1}}\mathscr{E}_{n-1}(Q).\]
Note that
\[(\log Q)^{R}\mathscr{E}_{n-1}(Q)=(\log Q)^{R}\mathscr{E}_{n-1}^{(\idl{c}_{1}'; \idl{c}_{2}')}(Q)= \left\{\begin{array}{*{2}{c}} \exp(\idl{c}_{1}'\sqrt{\log Q}+R\log \log Q)& \mbox{if $n=3$,}\\ (\log Q)^{\idl{c}_{2}'+R}& \mbox{if $n\geq 4$,}\end{array}\right.\]
hence for some absolute constant $c_{0}$ we can choose $\idl{c}_{1}=\idl{c}_{1}'+c_{0}R$ and $\idl{c}_{2}=\idl{c}_{2}'+R$ and obtain
\[(\log Q)^{R}\mathscr{E}_{n-1}^{(\idl{c}_{1}'; \idl{c}_{2}')}(Q)\ll \mathscr{E}_{n-1}^{(\idl{c}_{1}; \idl{c}_{2})}(Q).\]
This completes the proof of theorem \ref{mainthm}.

\renewcommand{\refname}{References}

{\sc Mathematisches Institut, Georg-August Universit\"at G\"ottingen, Bunsenstrasse 3-5, 37073 G\"ottingen, Germany}

{\it Email address:} {\tt florian.munkelt@mathematik.uni-goettingen.de}

\end{document}